\documentclass[pdftex,a4paper]{amsart}
\usepackage[utf8]{inputenc}
\usepackage[T1]{fontenc}

\usepackage{amssymb}
\usepackage{amsfonts}
\usepackage{amsthm}
\usepackage{calrsfs}
\usepackage{array}
\usepackage{bbm}
\usepackage{stmaryrd}

\usepackage{csquotes}

\usepackage{mathtools}
\usepackage[svgnames]{xcolor}
\usepackage{etoolbox}
\usepackage{enumitem}
\usepackage[unicode,plainpages=false]{hyperref}

\newcommand{\m}[1]{
\ifdefequal{#1}{1}
{\mathbbm{#1}}
{\mathbb{#1}}
}

\newcommand{\q}[1]{\mathcal{#1}}

\newcommand{\mr}[1]{\mathrm{#1}}

\newcommand{\ds}{\displaystyle}
\newcommand{\be}{\begin{gather}}
\newcommand{\ee}{\end{gather}}
\newcommand{\ba}{\begin{align*}}
\newcommand{\ea}{\end{align*}}

\newcommand{\e}{\varepsilon}

\newcommand{\tendf}{\rightharpoonup}
\newcommand{\imp}{\Longrightarrow}

\newcommand{\loc}{\mathrm{loc}}

\renewcommand{\le}{\leqslant}
\renewcommand{\ge}{\geqslant}

\numberwithin{equation}{section}

\theoremstyle{plain}
\newtheorem{thm}{Theorem}[section]
\newtheorem*{thm*}{Theorem}

\newtheorem{prop}[thm]{Proposition}
\newtheorem{cor}[thm]{Corollary}
\newtheorem{lem}[thm]{Lemma}

\theoremstyle{definition}

\theoremstyle{remark}
\newtheorem{nb}{Remark}
\newtheorem{claim}[thm]{Claim}

\allowdisplaybreaks

\bibliography{../../references.bib}

\begin{document}

\title{Soliton resolution for equivariant wave maps to the sphere}

\author{Raphaël Côte}

\subjclass[2010]{35L05,35L71}

\keywords{wave maps, equivariant, classification, profile, soliton resolution}

\thanks{The author gratefully acknowledges the support of the European Research Council under the project ``Blow up, Dispersion and Solitons''.}

\begin{abstract}
We consider finite energy corotationnal wave maps with target manifold $\m S^2$. We prove that for a sequence of times, they decompose as a sum of decoupled harmonic maps in the light cone, and a smooth wave map (in the blow case) or a linear scattering term (in the global case), up to an error which tends to 0 in the energy space.
\end{abstract}

\maketitle

\section{Introduction}

\subsection{Statement of the main results}

Let $(M,g)$ be a Riemannian manifold, and $\m R^{1+d}$ be endowed with the Minkowski metric $\eta = \mathrm{diag}(-1, 1,\dots, 1)$. Wave maps $U: (\m R^{1+d}, \eta) \to (M, g)$ are defined formally as critical points of the Lagrangian 
 \begin{align*}
 \q L(U, \partial U) = \frac{1}{2} \int_{\m R^{1+d}} \eta^{\alpha \beta} \langle \partial_{\alpha} U,\partial_{\beta}U  \rangle_g dxdt.
 \end{align*}
In local coordinates, they satisfy the Euler-Lagrange equation
\begin{gather} \label{eq:wm}
 \begin{cases}
\Box U^k =  - \eta^{\alpha \beta} \Gamma^k_{i j}(U) \partial_{\alpha} U^i \partial_{\beta} U^j \\
 (U, \partial_t U)\vert_{t=0}=(U_0, U_1),
\end{cases}
\end{gather}
where $\Gamma_{i j}^k$ are the Christoffel symbols on $TM$. 

We refer to the review article \cite{Kri08} and the reference therein for recent developments regarding general wave maps.

We consider the case where $d=2$ and $M$ is a 2 dimensional surface of revolution with metric
\[ ds^2 = d\rho^2 + g(\rho)^2 d\theta, \]
where $(\rho, \theta)$ are the polar coordinates on $M$, and $g \in \q C^3(\m R)$. 

We assume that $U$ has corotationnal equivariant symmetry, that is, denoting $(r,\omega)$ the polar coordinates on $\m R^2$, it takes the form
\[ U(t,r,\omega) = (\psi(t,r),\omega). \]
for some function $\psi$. System \eqref{eq:wm} then simplifies to the following equation on $\psi$ :
\begin{gather}  \label{wm} \tag{WM}
\begin{cases}
\ds \partial_{tt} \psi - \partial_{rr} \psi - \frac{1}{r} \partial_r \psi + \frac{f(\psi)}{r^2} = 0\\
(\psi, \partial_t \psi)\vert_{t=0} = (\psi_0, \psi_1)
\end{cases} \quad \text{where } f=gg'.
\end{gather}
We say that such a solution $\vec \psi = (\psi, \partial_t \psi)$ to \eqref{wm} is a wave map.

We define the energy space $\q H \times L^2$ and similarly the Hilbert space $H \times L^2$ as follows: given a couple of function $\vec \phi = (\phi_{0}, \phi_1)$, and for $0 \le r_1 < r_2 \le \infty$,
\begin{align*}
E(\vec \phi; r_1,r_2) & := \int_{r_1}^{r_2} \left( |\phi_1(t,r)|^2 +  |\partial_r \phi_0(t,r)|^2 + \frac{|g(\phi_0(t,r))|^2}{r^2} \right) rdr, \\
\| \phi_0 \|_{H([r_1,r_2])}^2 & : = \int_{r_1}^{r_2} \left( |\partial_r \phi_0(r)|^2 + \frac{|\phi_0(r)|^2}{r^2} \right) rdr, \\
\| \vec \phi \|_{H \times L^2([r_1,r_2])} & = \int_{r_1}^{r_2} \left( | \phi_1(r)|^2 +  |\partial_r \phi_0(r)|^2 + \frac{|\phi_0(r)|^2}{r^2} \right) rdr.
\end{align*}
We omit $r_1,r_2$ in the case $r_1=0$ and $r_2=\infty$: the energy is  $E(\vec \phi)  := E(\vec \phi; 0 ,\infty)$, and $\q H \times L^2 = \{ \vec \phi \mid E(\vec \phi) < +\infty \}$.

If $\vec \psi = (\psi(t), \partial_t \psi(t))$ is a finite energy wave map, then at least formally its energy is preserved: for all $t$ where defined,
\begin{gather} \label{eq:energy=const}
E(\vec \psi(t)) = E(\vec \psi(0)).
\end{gather}
\eqref{wm} is energy critical in the following sense. Consider the scaling (for $\lambda >0$)
\[  \Lambda[\lambda]\vec \phi(t,r) := \left( \phi_0 \left( \frac{t}{\lambda}, \frac{r}{\lambda} \right), \frac{1}{\lambda} \phi_1 \left( \frac{t}{\lambda}, \frac{r}{\lambda} \right) \right). \]
Then
$\vec \psi$ is a wave map if and only if $  \Lambda[\lambda]\vec \phi$ is a wave map, and the energy is scaling invariant:
\[ E(\vec \psi) = E(  \Lambda[\lambda]\vec \phi). \]
Notice that the $H \times L^2$ norm is also scaling invariant.

Recall that if $\phi \in \q H$, then $\phi$ continuous and bounded, and has well defined limits at $0$ and $+\infty$, which cancel $g$: we denote them $\phi(0)$ and $\phi(\infty)$. If $\vec \phi$ is a wave map, these limits do not depend on time. 

This motivates the introduction of the set  of points where $g$ vanishes
\[ \q V := \{ \ell \in \m R \mid g(\ell) =0 \}. \]
Also, let
\[ \ds G(x) := \int_0^x |g(y)| dy. \]

We recall the local well-posedness result in the energy space, due to Shatah and Tahvildar-Zadeh.
\begin{thm*}[{\cite{STZ94}}]
Let $(\psi_0, \psi_1) \in \q H \times L^2$. Then there exists a unique wave map  $\vec\psi = (\psi, \partial_t \psi) \in \q C(I, \q H \times L^2)$ solution to \eqref{wm}, defined on a maximal interval $I =: (T^-(\vec \psi), T^+(\vec \psi))$, and which preserves the energy \eqref{eq:energy=const}.
\end{thm*}

The wave map equation \eqref{wm} has been intensively studied as a model for geometric wave equations. It has been long understood that the geometry of the target $M$, i.e. the metric $g$, plays a crucial role in the long time behavior of wave maps. Let us mention the result by Struwe \cite{Str03}: a wave map that blows up in finite time must bubble up a harmonic map at blow up time. In particular, if $M$ does not admit non constant harmonic maps, then any wave map is global in time.

\medskip

Actual examples of wave maps blowing up in finite time were \emph{constructed} by Rodnianski and Sterbenz \cite{RS10} and Raphaël and Rodnianski \cite{RR12} (as a perturbation of the self similar regime), and by Krieger, Schlag and Tataru \cite{KST08} (with prescribed, polynomial blow up rate).

On a different side, together with Kenig, Lawrie and Schlag \cite{CKLS13a,CKLS13b}, we \emph{classified} the asymptotic behavior of wave maps with energy less than 3 times the energy of a harmonic map, for large time or near blow up time.

\medskip

Our goal in this paper is to obtain a similar classification for wave maps of arbitrarily large energy, that is to relax the bound on the energy. We provide a description of a wave map into decoupled profiles, a so called soliton resolution.

It turns out that these profiles are harmonic maps and linear scattering terms. Recall that a harmonic map is a solution $Q$ of finite energy of
\[ \partial_{rr} Q  + \frac{1}{r} \partial_r Q = \frac{f(Q)}{r^2}. \]
(Hence $(Q,0)$ is a finite energy stationary wave map). From \cite{C05}, they are classified as follows: a non constant harmonic map is monotonic, satisfies one of the ODEs
\[ r \partial_r Q = g(Q) \quad \text{or} \quad r \partial_r Q = -g(Q), \]
and joins two consecutive points of $\q V$, that is for some $\ell, m \in \q V$, $\ell < m$,
\[ \{ Q(0), Q(\infty) \} = \{ \ell, m \}  \quad \text{and} \quad \q V \cap (\ell,m) = \varnothing. \]
It has energy $E(Q):=E(Q,0)=2(G(m) - G(\ell))$. In particular if $\q V =1$, there exists no non constant harmonic map (if $\q V$ is empty, there is no finite energy map).

On the other hand, given $\ell \in \q V$, we define the linearized wave map flow around $\ell$:
\begin{gather} \label{LWl} \tag{LW$_\ell$}
\partial_{tt} \phi - \partial_{rr} \phi -\frac{1}{r} \partial_r \phi + \frac{g'(\ell)^2}{r^2} \phi = 0.
\end{gather}
Solutions to this linear wave equation preserve the following $H \times L^2$ related quantity
\[ \| \vec \phi(t) \|_{H_\ell \times L^2}^2 : = \int \left( |\partial_t \phi(t,r)|^2 + | \partial_r \phi(t,r)|^2 + \frac{g'(\ell)^2 \phi(t,r)^2}{r^2} \right) rdr = \| \vec \phi(0) \|_{H_\ell \times L^2}^2. \]

We now state the main result of this paper. For this, we make the following assumptions on the metric $g$:
\begin{enumerate}
\item[(A1)] $G(x) \to \pm\infty$ as $x \to \pm \infty$. 
\item[(A2)] $\q V$ is discrete,
\item[(A3)] For all $\ell \in \q V$, $g'(\ell) \in \{ -1,1 \}$.
\end{enumerate}
Assumption (A1) prevents the formation of bubbles at infinity, and is a very natural assumption. (A2) is also a natural assumption of non degeneracy of $g$, which prevent a decomposition with harmonic maps of arbitrarily small energy. 

The physically relevant metrics $g$ are
\begin{enumerate}
\item $g(\rho) = \sin(\rho)$ (wave maps to the sphere $\m S^2$), and
\item $g(\rho) = 1-\rho^2$ (radial 4D Yang-Mills equation).
\end{enumerate}
Hence (A3) allows to handle wave maps to the sphere $\m S^2$; however, dealing with the radial 4D Yang-Mills equation requires to relax (A3) to
\begin{enumerate}
\item[(A3')] For all $\ell \in \q V$, $g'(\ell) \in \{ -2,-1,1,2 \}$.
\end{enumerate}
It should be noted that most of the results in this article hold under (A3') instead of (A3). One could even consider the natural condition $g'(\q V) \subset \m Z \setminus \{ 0 \}$,
which makes the linearized problem \eqref{LWl} around $\ell$ to be of wave type (in dimension $2 g'(\ell)+2$). However large $g'(\ell)$ raise technical issues for the Cauchy problem as noted in \cite[Theorem 2]{CKM08}, which is restricted to the case (A3'); if these issues could be dealt with, most results here hold under this last condition.

\begin{thm} \label{th1}
We make assumptions (A1)-(A2)-(A3).

Let $\vec \psi(t)$ be a finite energy wave map. Then there exist a sequence of time $t_n \uparrow T^+(\vec\psi)$, an integer $J \ge 0$, $J$ sequences of scales $\lambda_{J,n} \ll \cdots \ll \lambda_{2,n} \ll \lambda_{1,n}$ and $J$ harmonic maps $Q_1, \dots, Q_J$ such that 
\[ Q_J(0) = \psi(0), \quad Q_{j+1}(\infty) = Q_{j}(0) \quad \text{for } j=1, \dots, J-1, \]
and that the following holds.
\begin{enumerate}
\item If $T^+(\vec \psi) = +\infty$, denote $\ell = \psi(\infty)$. Then $\lambda_{1,n} \ll t_n$ and there exists a solution $\vec \phi_L(t) \in \q C(\m R, H \times L^2)$ to the linear wave equation \eqref{LWl} such that 
\begin{gather} \label{dec_g}
\vec \psi(t_n) =  \sum_{j=1}^J \left(Q_{j} \left(\cdot/\lambda_{j,n}\right) - Q_j(\infty), 0\right) + (\ell,0) + \vec \phi_L(t_n) +  \vec b_n,
\end{gather}
where $\vec b_n \to 0$ in $H \times L^2$ as $t \to \infty$ and $Q_1(\infty) = \ell$.
\item If $T^+(\vec \psi) < +\infty$, then $\lambda_{1,n} \ll T^+(\vec \psi)-t_n$ and there exists a function $\vec \phi \in \q H \times L^2$ of finite energy such that
\begin{gather}
\label{dec_b}
\vec \psi(t_n) =  \sum_{j=1}^J \left(Q_j \left(\cdot/\lambda_{j,n}\right) - Q_j(\infty) , 0\right) + \vec \phi +  \vec b_n,
\end{gather}
where $\vec b_n \to 0$ in $H \times L^2$ as $t \to \infty$, $\phi(0) =\lim_{t \uparrow T^+(\vec\psi)} \psi(t, T^+(\vec\psi)-t)$ and $Q_1(\infty) = \phi(0)$.
\end{enumerate}
\end{thm}

\begin{nb}
Notice that in the global case, $E(\vec \psi) =  \sum_{j=1}^J E(Q_j) + \| \vec \phi_L \|_{H_\ell \times L^2}^2$, 
and in the blow up case, $E(\vec \psi) =  \sum_{j=1}^J E(Q_j) + E(\vec \phi)$.
This gives a bound on $J$.

Also, if $T^+(\vec\psi) = +\infty$, then $J \ge \# (\q V \cap [\psi(0),\psi(\infty)])-1$. This last number can be made arbitrarily large when $\q V$ is infinite (as for $g=\sin$).
\end{nb}

\begin{nb}
The question whether the decomposition holds for all times and not merely for a sequence is open. However
there are some cases where it can be proved. For example, when the excess energy of $\vec \psi$ with respect to the energy to connect $\psi(0)$ to $\psi(\infty)$ is not enough to bubble more harmonic maps: that is
\[ E(\vec \psi) < 2 |G(\psi(\infty)) -G(\psi(0))| + 4 \delta, \]
where $\delta = \min_{\ell \in \q V \setminus \{ \psi(0), \psi(\infty) \} } \{ |G(\psi(\ell)) -G(\psi(0))|, |G(\psi(\ell) -G(\psi(\infty))| \}$.  We refer to \cite[Proof of Theorem 1.3]{CKLS13a} for a detailed argument in the case $J=1$.
\end{nb}

\begin{nb}
Many possibilities are left open regarding the behavior of the $\lambda_{j,n}$: for example, in the global case, one could have $\lambda_{J,n} \to 0$ (infinite time blow-up), or $\lambda_{J,n} \to +\infty$ (infinite time flattening). Although no such solutions were constructed for \eqref{wm}, let us refer to \cite{DK12} in the context of  the semilinear wave equation.
\end{nb}

Theorem \ref{th1} is an extension of \cite{CKLS13a,CKLS13b} where only one profile was allowed (i.e $J=1$) through the bound on the energy. It is in the spirit of the seminal papers by Duyckaerts, Kenig and Merle \cite{DKM11,DKM12a,DKM12b,DKM13} where large solutions of the radial energy critical (focusing) wave equation in 3D were described. Let us however observe that their analysis did not encompass type I blow up solutions (i.e when $\limsup_{t \uparrow T^+(\vec u)} \| \vec u (t) \|_{\dot H^1 \times L^2} = +\infty$). 
This phenomenon does not occur in the wave map case (mainly because the energy is coercive, even if it doesn't bound $H \times L^2$), and we give a description of any wave map, without any further assumption.

\subsection{Outline of the proof}

Let us mention two delicate issues. First, geometry has to be taken into account: the harmonic maps do never belong to $H \times L^2$. This means we must derive a procedure to extract them without relying on a linear profile decomposition, as for the wave equation. 

\medskip

Second, the linearized operator of the wave map flow at spatial infinity is of wave type in \emph{even} dimension: most of the delicate linear estimates available in the radial 3D case break down for wave maps, in particular the so-called ``energy channels''. When $g'(\ell)$ is odd, \eqref{LWl} corresponds to a wave equation in $\m R^d$ with $d \equiv 0 \mod 4$, and the linear estimate obtained in \cite{CKS13} in this case suffices to conclude. When $g'(\ell)$ is even, then $d \equiv 2 \mod 4$ and the desired linear estimate fails. This is the reason why we must restrict ourselves to (A3) instead of (A3').

\bigskip

The first step in the proof is to choose a sequence of time $t_n \to T^+(\vec \psi)$ on which the space-time kinetic energy inside the light cone vanishes. This is a reformulation that the averaged kinetic energy inside the light cone vanishes, which is a well known result, and is the content of Section 2.1. Section 2.2 focuses on various aspects of the profile decomposition in $H \times L^2$ to be used later in the paper.

\medskip

The second step, Section 3, is concerned with sequence of wave maps whose space-time kinetic energy vanishes, and shows, in Theorem \ref{th:profiles}, that up to a subsequence, one can construct a bubble decomposition i.e extract the harmonic maps. This decomposition holds up to an error which tends to 0 in $L^\infty$. Notice that this result does not make use of assumption (A3) or (A3'), but only (A1) and (A2).

\medskip

The bound on the error is insufficient to capture the linear scattering term for example, but it is enough to derive a sharp scattering theorem \emph{below the threshold in} $L^\infty$. As linear scattering is involved, we do need assumption (A3') here. This result has its own interest: let us state it here, and postpone the proof to Section 4.

\medskip
 
For $\ell \in \q V$, define $d_\ell$ as the distance of $\ell$ to the closest (distinct) element in $V$:
\[ d_\ell = \inf \{ |\ell - k | \mid k \in \q V \setminus \{ \ell \} \}. \]
$d_\ell>0$ due to assumption (A2).

\begin{thm} \label{th:scat}
Let $\ell \in \q V$ and assume (A1)-(A2), and
\[ g'(\ell) \in \{ -2,-1,1,2 \}. \]
Let $\vec \psi$ be a wave map such that $\psi(\infty) = \ell$, and that for some $c < \delta_\ell$,
\begin{equation} \label{hyp:infty_bound}
\forall t \in [0,T^+(\vec \psi)), \quad \| \psi(t) - \ell \|_{L^\infty} \le c < d_\ell.
\end{equation}
Then $T^+(\vec \psi) = +\infty$ and $\vec \psi$ scatters at $+ \infty$, in the sense 
\[ \| \psi - \ell \|_{S_\ell([0,+\infty))} < +\infty. \]
($S_\ell$ is an adequate Strichartz space, defined below on \eqref{def:S}). 
It follows that there exists a (unique) solution $\vec \phi_L$ to \eqref{LWl} such that
\[ \| \vec \psi(t) - (\ell,0) -  \vec \phi_L(t) \|_{H \times L^2} \to 0 \quad \text{as} \quad t \to +\infty. \]
\end{thm}

\begin{nb}
Observe that if $g$ vanishes in at most one point $\ell$, then $d_\ell = +\infty$, and so Theorem \ref{th:scat} proves that all wave maps are global and \emph{scatter} in this case. This strengthens the global well posedness result by Struwe \cite{Str03} mentioned above. 
\end{nb}

As a consequence of Theorem \ref{th:scat}, we can extract the scattering term (for \emph{all times}, not merely a sequence) in the global case. In an analogous way, we can define the regular part $\vec \phi$ in the blow up case. This is the content of Propositions \ref{prop:scat_state} and \ref{prop:reg_part} of Section 5. Let us emphasize that this step only requires (A3'). 

\medskip

In Section 6,  we revisit Theorem \ref{th:profiles}. Under the additional assumption (A3) -- crucial but used only on this step, we show that the error term tends to 0 in $H \times L^2$. This section is independent of Sections 4 and 5.

\medskip

Finally, we gather all the previous results together in Section 7 and prove of Theorem \ref{th1}.

\section{Preliminaries}

The purpose of this section is to recall or adapt a few important earlier results, and derive some consequences.

\subsection{The self-similar region}

\subsubsection{Global wave maps}

Throughout this Subsection, let $\vec \psi$ be a finite energy  wave map such that $T^+(\vec \psi) = +\infty$.

\begin{prop}[{\cite[Proposition 2.1]{CKLS13b}}] \label{prop:sse}
For all $\lambda >0$,
\[ \limsup_{t \to +\infty} E(\vec \psi(t);\lambda t,t-A) \to 0 \quad \text{as} \quad A \to + \infty. \]
\end{prop}

\begin{proof}
The argument in \cite{CKLS13b} is done for $g =\sin$, after ideas of \cite{CTZ93}, and extends seamlessly for any non linearity $g$.
\end{proof}

We derive a few consequences from this.
One fundamental outcome of Proposition \ref{prop:sse} is that the kinetic part of the energy vanishes in an averaging sense. More precisely, we have:

\begin{cor}[{\cite[Corollary 2.2]{CKLS13b}}] \label{cor:psi_dot1}
\[ \limsup_{T \to +\infty} \frac 1 {T} \int_A^T \int_0^{t-A}  | \partial_t \psi(t,r)|^2 rdrdt \to 0 \quad \text{as} \quad A \to + \infty. \]
\end{cor}

From there, we find a sequence of times for which the condition (3) in Theorem \ref{th:profiles} holds.

\begin{cor} \label{cor:psi_t->0}
There exists a sequence $t_n \uparrow +\infty$ such that 
\[ \sup_{s, 0<s\le t_n/2} \frac{1}{s} \int_{t_n - s}^{t_n + s} \int_0^{t/2}  | \partial_t \psi(t,r)|^2 rdrdt \to 0 \quad \text{as} \quad n \to + \infty. \]
\end{cor}

\begin{proof}
Corollary \ref{cor:psi_dot1} shows that
\[ \limsup_{ T \to +\infty}  \frac 1 {T} \int_{2A}^T \int_0^{t-A}  | \partial_t \psi(t,r)|^2 rdrdt \to 0 \quad \text{as} \quad A \to + \infty, \]
hence, if we let $\ds f(t) = \int_0^{t/2}  | \partial_t \psi(t,r)|^2 rdr$, and as $t/2 \le t-A$ if $t \ge 2A$, we have
\[ \limsup_{ T \to +\infty}  \frac 1 {T} \int_{A}^T f(t) dt \to 0 \quad \text{as} \quad A \to + \infty. \]
We now argue by contradiction. Assume that the conclusion is not correct, then this means that for some $\delta>0$,
\[ \liminf_{T \to +\infty} \sup_{s, 0\le s \le T/2} \frac{1}{s} \int_{T -s}^{T + s} f(t) dt \ge 41 \delta. \]
Fix $A$ and $T_0 \ge 2A$ be large enough such that for all $T \ge T_0$,
\begin{itemize}
\item $\ds \frac 1 {T} \int_{A}^T f(t) dt \le \delta$
\item there exists $s(T) \in [0,T/2]$ such that $\ds \frac{1}{s(T)} \int_{T - s(T)}^{T + s(T)} f(t) dt \ge 40 \delta$. 
\end{itemize}
Consider the sets $(T - s(T), T+s(T))$ for $T \in [T_0, 2T_0]$. Their diameter is bounded by $2T_0$, hence Vitaly covering lemma applies: there exist a sequence $(T^n)_n$ such that  the intervals $(T^n - s(T^n ), T^n+s(T^n ))$ are disjoints and
\[  [T_0,2T_0] \subset \bigcup_{T \in [T_0,2T_0]}(T - s(T), T+s(T)) \subset \bigcup_{n} (T^n - 5s(T^n ), T^n+5s(T^n )). \]
From this last condition, it follows that
\[ T_0 \le 10 \sum_n s(T^n). \]
On the other hand, by definition of the $s(T^n)$, we get that
\[ \int_{T^n - s(T^n )}^{T^n+s(T^n )} f(t) dt \ge 40 \delta s(T^n), \]
and as the intervals under consideration are disjoint, we infer
\[ \int_{T_0/2}^{3T_0} f(t) dt \ge \sum_{n} \int_{T^n - s(T^n )}^{T^n+s(T^n )} f(t) dt \ge 40 \delta \sum_n s(T^n) \ge 4 \delta T_0. \]
But as $T_0/2 \ge A$, we also have
\[ \frac{1}{3T_0} \int_{T_0/2}^{3T_0} f(t) dt \le \delta, \]
and we reached a contradiction.
\end{proof}

Finally we recall that the $L^\infty$ norm outside vanishes in the self similar region and outside the light cone. 

\begin{cor}[{\cite[Corollary 2.3]{CKLS13b}}] \label{cor:lc_infty} 
For any $\lambda>0$ we have 
\begin{align*}
\| \psi(t) - \psi(\infty)\|_{L^{\infty}(r \ge \lambda t)} \to 0 \quad \text{as} \quad t \to \infty.
\end{align*}
\end{cor}

\begin{proof}
The argument in \cite{CKLS13b} is done for $g =\sin$ and can be extended seamlessly for any non linearity $g$.
\end{proof}

\subsubsection{Blow up wave maps}

Throughout this Subsection, let $\vec \psi$ be a finite energy  wave map such that $T^+(\vec \psi) < +\infty$. The results here very similar to those in the global case, and in fact simpler (integration can be done up to the light cone).

\begin{prop}[{\cite[Lemma 2.2]{STZ92}}] \label{prop:sse2}
For all $\lambda \in (0,1)$,
\[ E(\vec \psi(t);\lambda ( T^+(\vec \psi)-t), T^+(\vec \psi)-t) \to 0 \quad \text{as} \quad t \uparrow 1. \]
\end{prop}

\begin{cor}[{\cite[Corollary 2.2]{STZ92}}] \label{cor:psi_dot2}
\[ \frac 1 {T^+(\vec \psi)-T} \int_T^{T^+(\vec \psi)} \int_0^{T^+(\vec \psi)-t}  | \partial_t \psi(t,r)|^2 rdrdt \to 0 \quad \text{as} \quad T \to T^+(\vec \psi). \]
\end{cor}

\begin{cor} \label{cor:psi_t->02}
There exists a sequence $t_n \uparrow T^+(\vec \psi)$ such that 
\[ \sup_{s, 0<s\le T^+(\vec \psi)-t_n} \frac{1}{s} \int_{t_n - s}^{t_n + s} \int_0^{T^+(\vec \psi)-t}  | \partial_t \psi(t,r)|^2 rdrdt \to 0 \quad \text{as} \quad n \to + \infty. \]
\end{cor}

\begin{proof}
It is very similar to the proof of Corollary \ref{cor:psi_t->0}. Let 
\[ f(t) = \int_0^{T^+(\vec \psi)-t}  | \partial_t \psi(t,r)|^2 rdr \to 0. \]
$f(t) \ge 0$ and we know that 
\[ \frac{1}{T^+(\vec \psi)-T} \int_T^{T^+(\vec \psi)} f(t) dt \to 0 \quad \text{as} \quad T \to T^+(\vec \psi). \]
Assume that the conclusion fails, for the sake of contradiction: then there exists $\delta>0$ and a function  $s$ defined on $[T^+(\vec \psi)-\delta,T^+(\vec \psi))$ such that $0 \le s(T) \le T^+(\vec \psi)-T$ and
\begin{gather} \label{eq:contr3}
\frac{1}{s(T)} \int_{T-s(T)}^{T+s(T)} f(t)dt \ge 40 \delta.
\end{gather}
Corollary \ref{cor:psi_dot2} yields $T_0<T^+(\vec \psi)$ such that for all $T \in [2T_0-T^+(\vec \psi),T^+(\vec \psi))$, 
\begin{gather} \label{eq:contr4}
\frac{1}{T^+(\vec \psi)-T} \int_T^{T^+(\vec \psi)} f(t) dt \le \delta.
\end{gather}
We apply Vitali covering lemma to the intervals $(T-s(T), T+s(T))$ where $T \in [T_0,T^+(\vec \psi))$ to find a sequence $T^n$ such that $(T^n-s(T^n), T^n+s(T^n))$ are disjoints and
\[ [T_0,T^+(\vec \psi)) \subset \bigcup (T^n - 5 s(T^n), T^n + 5s (T^n)). \]
Hence taking the length:
\[ T^+(\vec \psi)-T_0 \le 10 \sum_n s(T^n). \]
Therefore, as the interval are disjoint, and using \eqref{eq:contr3},
\[ \int_{2T_0-T^+(\vec \psi)}^{T^+(\vec \psi)} f(t) dt \ge \sum_n \int_{T^n -s(T^n)}^{T^n + s(T^n)} f(t) dt \ge 4 \delta (T^+(\vec \psi)-T_0). \]
Now $T^+(\vec \psi)-(2T_0-1) = 2 (T^+(\vec \psi)-T_0)$, and we reached a contradiction with \eqref{eq:contr4}.
\end{proof}

\subsection{Energy concentration on the light cone for \texorpdfstring{\eqref{LWl}}{(LWl)}}

For the rest of this Subsection, we fix $\ell \in \q V$ and focus on the linear equation \eqref{LWl}.

We define the transformation $\q T$ defined by
\[ (\q T\phi)(r) = \phi(r)/ r^{g'(\ell)}. \]
Then  $\vec \phi$ is a solution of $\eqref{LWl}$ if and only if $\vec \varphi = (\q T \phi, \q T \partial_t \phi)$ solves the radial wave equation in $2+2g'(\ell)$ dimensions:
\begin{equation} \label{LWkD}
\partial_{tt} \varphi - \partial_{rr} \varphi - \frac{1+2 g'(\ell)}{r} \partial_r \varphi =0.
\end{equation}

Observe that 
\[ \| \phi \|_{H_\ell} = \| \q T \phi \|_{\dot H^1(r^{1+2g'(\ell)}dr)}. \]
 The norms $H_\ell$ and $H$ are equivalent. It follows that $\q T$ is a bicontinuous bijective linear map $L^2(rdr) \to L^2(r^{1+2g'(\ell)} dr)$ and $H \to \dot H^1(r^{1+2g'(\ell)} dr)$.

We start by recalling a result regarding equipartition of energy and concentration of energy on the light cone for linear solutions.

\begin{prop} \label{prop:light_cone_en_conc}
Let $\vec \phi$ be a solution to \eqref{LWl}. Then
\[ \limsup_{t \to +\infty} \| \vec \phi(t) \|_{H \times L^2(|r-t|\ge A)} \to 0 \quad \text{as} \quad A \to +\infty. \]
Also,
\[ \| \partial_t \phi(t) \|_{L^2}^2 \to \frac{1}{2} \| \vec\phi(t) \|_{H_\ell \times L^2}^2, \quad \| \phi(t) \|_{H_\ell}^2 \to \frac{1}{2} \| \vec\phi(t) \|_{H_\ell \times L^2}^2. \]
\end{prop}

\begin{proof}
The first statement is the content of \cite[Theorem 4]{CKS13}. The second is equipartition of the energy, and is classical for the linear wave equation.
\end{proof}

We will sometimes use, in the context of a profile decomposition, the following weaker form, namely all the energy concentrates on one scale.

\begin{cor} \label{cor:prof_energy_scale}
Let $\vec \phi$ be a solution to \eqref{LWl}, and $(t_n, \lambda_n)$ be two sequences with $\lambda_n >0$, and such that
\[ \frac{|t_n|}{\lambda_n} \to +\infty. \]
Then for any $c >1$, as $n \to +\infty$,
\[ \left\| \left( \phi \left( - \frac{t_n}{\lambda_n} , \frac{r}{\lambda_{j,n}} \right), \frac{1}{\lambda_{j,n}} \partial_t \phi \left( - \frac{t_n}{\lambda_n} , \frac{r}{\lambda_{j,n}} \right) \right) \right\|_{H_\ell \times L^2(\frac{1}{c} t_n \le r \le ct_n)} \to \| \vec \phi(0) \|_{H_\ell \times L^2}. \]
\end{cor}

Now we recall a result giving some condition so that some energy of a linear solution remains outside of the light cone. It was already crucial in \cite{CKLS13a,CKLS13b}. Here is the only place in the argument where we need to restrict to odd $g'(\ell)$ (and hence (A3) to have the nonlinear argument run).

\begin{prop}[{\cite[Theorem 1]{CKS13}}] \label{prop:en_conc}
Assume $g'(\ell)$ is an odd integer. There exists $\beta(\ell) >0$ such that the following holds. 
Let $\vec \phi$ be a solution to \eqref{LWl}, such that
\[ \partial_t \phi (0) =0. \]
Then for all $t \in \m R$,
\[ \| \vec \phi(t) \|_{H \times L^2(r \ge |t|)}^2 \ge \beta(\ell) \| \vec \phi(0) \|_{H \times L^2}. \]
\end{prop}

\begin{proof}
As mentioned, this is an easy consequence of \cite[Theorem 1]{CKS13} and the remark that follows. We refer to \cite[Corollary 2.3]{CKLS13a} for the complete argument to pass from the linear wave equation to \eqref{LWl}. 
\end{proof}

\subsection{Profile decomposition for \texorpdfstring{\eqref{LWl}}{(LWl)} in \texorpdfstring{$H \times L^2$}{HxL2}}

Again, we fix $\ell \in \q V$ for the rest of this subsection.

Our goal is to derive a suitable notion of linear profile decomposition in the spirit of \cite{BG99}, adapted to our setting, in particular $L^\infty$ bounds.

Using transformation $\q T$, a notion of profile decomposition for the wave equation will immediately translate to a similar decomposition for \eqref{LWl}, but we will in fact improve it. We elaborate on this in what follows.

For $I$ an interval of $\m R$, we define 
\begin{gather} \label{def:S}
 \| \phi \|_{S_\ell(I)} := \left( \int_{t \in I} \int_{r=0}^{\infty} |\phi(t,r)|^{2+3/g'(\ell)} \frac{drdt}{r^2} \right)^{\frac{1}{2+3/g'(\ell)}}.
 \end{gather}
It is simply the norm of $\q T\phi$ in the Strichartz space $L^{2+3/g'(\ell)}_{t,r}(r^{1+2g'(\ell)}drdt)$, adap\-ted to the $\dot H^1$ critical wave equation in dimension $2 + 2g'(\ell)$ (we refer to \cite[Section 3]{CKM08} for further details).

\begin{lem}
Let $\ds \theta(\ell) = \frac{3}{4+6/g'(\ell)} \in (0,1)$. There exist $C >0$ such that for any finite energy solution $\vec \gamma$ to the linear wave equation \eqref{LWl}. 
\[ \| \gamma \|_{L^\infty_t (\m R, L^\infty_r)} \le  C \| \vec \gamma(0) \|_{H \times L^2}^{\theta(\ell)} \| \gamma \|_{S_\ell(\m R)}^{1-\theta(\ell)}. \]
\end{lem}

\begin{proof}
As $\vec \gamma$ is a finite energy solution to the linear equation \eqref{LWl}, all terms in the desired estimate are finite.

Denote $M = \| \gamma \|_{L^\infty_t (R, L^\infty_r)}$ and $A = \| \gamma(0) \|_H$. Notice that $\| \vec \gamma(t) \|_{H \times L^2}$ is bounded below and above by $\| \vec \gamma(0) \|_{H \times L^2}$ (because it is equivalent to the conserved quantity $\| \vec \gamma (t) \|_{H_\ell \times L^2}$). It follows that  for some $K$ only depending on $\ell$, $M \le KA$.

As all the functional space under consideration are invariant under scaling and time translation, we can assume that  $|\gamma(0,1)| \ge 2M/3$.

Let $s,t \in \m R$, and $r \ge q \ge  0$. Then for any $\rho \in [q,r]$ we have
\begin{align*}
| \gamma(t,r) - \gamma(s,q)|^2 & \le \left( \int_q^{\rho} |\partial_r \gamma(s,r')| dr' + \int_{[s,t]} | \partial_t \gamma(t',\rho)| dt' +  \int_{\rho}^r | \partial_r \gamma(t,r')| dr' \right)^2 \\
& \le 3 \ln \frac{\rho}{q} \int_q^{\rho} |\partial_r \gamma(s,r')|^2 r' dr' + 3 \frac{|t-s|}{\rho} \int_{[s,t]} | \partial_t \gamma(t',\rho)|^2 \rho dt' \\
&\qquad  + 3 \ln\frac{r}{\rho} \int_{\rho}^r | \partial_r \gamma(t,r')|^2  r'dr'  \\
& \le 3 \ln \frac{r}{q} \| \gamma(s) \|_H^2 + 3 \frac{|t-s|}{q} \int_{[s,t]} | \partial_t \gamma(t',\rho)|^2 \rho dt' 
\end{align*}
After averaging in $\rho \in [q,r]$, it transpires
\begin{align*}
| \gamma(t,r) - \gamma(s,q)|^2 & \le C \ln \frac{r}{q} \| \vec \gamma(0) \|_{H\times L^2}^2 + 3 \frac{|t-s|}{q|r-q|} \int_q^r \int_{[s,t]} | \partial_t \gamma(t',\rho)|^2 \rho dt' d\rho \\
& \le C \left( \ln \frac{r}{q} + \frac{|t-s|^2}{q|r-q|} \right)  \| \vec \gamma(0) \|_{H\times L^2}^2.
\end{align*}
Pick $q=1$, $s=0$, and fix $B = \max (K/3C,2)$. We deduce, for all $r \in [1,1+B]$ and $|t| \le \sqrt{B(r-1)}$,
\[ | \gamma(t,r)| \ge 2M/3 - CA \sqrt{\ln (1+B) + B} \ge 2M/3 - A BC \ge M/3. \]
$B = M/(3AC) \le 1/3C$ universal constant. Hence
\begin{align*}
\| \gamma \|_{S(\m R)}^{2+3/g'(\ell)} &  \ge \int_1^{1+B} \int_{|t| \le \sqrt{B(r-1)}} ( M/3)^{2+3/g'(\ell)} \frac{drdt}{r^2} \\
& \ge (M/3)^{2+3/g'(\ell)} \int_1^{1+B} \frac{\sqrt{B(r-1)}}{r^2} dr \ge C B^{3/2}  (M/3)^{2+3/g'(\ell)} \\
& \ge \frac{M^{7/2+3/g'(\ell)}}{C A^{3/2}}.
\end{align*}
We can conclude, with $\ds \theta = \frac{3}{4+6/g'(\ell)}$,  $M \le C A^{\theta} \| \gamma \|_{S(\m R)}^{1-\theta}$.
\end{proof}

From this, the profile decomposition for \eqref{LWl} takes the following form.

\begin{thm}[Profile decomposition] \label{th:prof_decomp}
Let $\vec \psi_n$ a bounded sequence of $H \times L^2$. Then there exists a sequence of scales $(t_{j,n},\lambda_{j,n})_n$ and linear profiles $\vec V_{j,L}$ solutions to \eqref{LWl}, such that, up to a subsequence that we still denote $\vec \psi_n$, we have for all $J\ge 1$
\[ \vec \psi_n(r) = \sum_{j=1}^J \left( V_{j,L} \left( - \frac{t_{j,n}}{\lambda_{j,n}}, \frac{r}{\lambda_{j,n}} \right), \frac{1}{\lambda_{j,n}} \partial_t V_{j,L} \left( - \frac{t_{j,n}}{\lambda_{j,n}}, \frac{r}{\lambda_{j,n}} \right) \right) + \vec \gamma_{J,n}(0), \]
where $\vec \gamma_{J,n}$ is a solution to \eqref{LWl} which satisfies
\[ \limsup_n \| \gamma_{J,n} \|_{S_\ell(\m R)} + \| \gamma_{J,n} \|_{L^\infty_{t,r}} \to 0 \quad \text{as} \quad J \to + \infty, \]
 and for all $j$,
\[ \forall n, t_{j,n} =0 \quad \text{ or } \quad \left( \frac{t_{j,n}}{\lambda_{j,n}} \text{ has a limit which is } +\infty \text{ or } -\infty \right), \]
and if $j \ne k$,
\[ \frac{\lambda_{j,n}}{\lambda_{k,n}} + \frac{\lambda_{k,n}}{\lambda_{j,n}} \to + \infty, \quad \text{or} \quad \left( \forall n, \lambda_{j,n} = \lambda_{k,n} \text{ and } \frac{|t_{j,n} - t_{k,n}|}{\lambda_{j,n}} \to +\infty \right). \]
Furthermore, there hold
\begin{enumerate}
\item Pythagorean expansion (and equipartition) of the $H_\ell \times L^2$ norm: for all fixed $J$, we have as $n \to +\infty$
\begin{align*} 
\|  \psi_{0,n} \|_{H_\ell}^2 & =  \sum_{j=1}^J \left\| V_{j,L} \left( - \frac{t_{j,n}}{\lambda_{j,n}}, \frac{r}{\lambda_{j,n}} \right) \right\|^2_{H_\ell} + \| \gamma_{J,n}(0,r) \|^2_{H_\ell} + o_{n}(1), \\
\|  \psi_{1,n} \|_{L^2}^2 & =  \sum_{j=1}^J \left\| \frac{1}{\lambda_{j,n}} \partial_t V_{j,L} \left( - \frac{t_{j,n}}{\lambda_{j,n}}, \frac{r}{\lambda_{j,n}} \right) \right\|^2_{L^2} + \| \partial_t \gamma_{J,n}(0,r) \|^2_{H_\ell} + o_{n}(1).
\end{align*}
\item Pythagorean expansion of the energy: for all $J$ fixed, we have as $n \to +\infty$
\begin{equation*}
E(\vec \psi_n) = \sum_{j=1}^J E \left( V_{j,L} \left( - \frac{t_{j,n}}{\lambda_{j,n}} \right), \frac{1}{\lambda_{j,n}} \partial_t  V_{j,L} \left( - \frac{t_{j,n}}{\lambda_{j,n}} \right) \right) + E(\vec \gamma_{J,n}(0)) + o(1).
\end{equation*}
\item $L^\infty$ profile selection:  $\| \psi_n \|_{L^\infty}$ has a limit as $n \to +\infty$ and exists
\begin{equation} \label{sup_psi_n}
\lim_n \| \psi_n \|_{L^\infty} = \sup_{j \in \q J_0}  \| V_j(0) \|_{L^\infty} \quad \text{where} \quad  \q J_0 := \{ j \ge 1 \mid  \forall n, \ t_{j,n} =0 \}.
\end{equation}
(with the convention that the $\sup$ is 0 if $\q J_0$ is empty).
\end{enumerate}
\end{thm}

\begin{proof}
This is essentially contained in \cite[Main Theorem]{BG99}. We refer to \cite[Corollary 2.15]{CKLS13a}  for the profile decomposition in the wave map context and the Pythagorean expansion of the $H \times L^2$ norm, and to \cite[Lemma 2.16]{CKLS13a} for the Pythagorean expansion of the energy.

The only extra points with respect to the usual profile decomposition are
\[  \| \gamma_{J,n} \|_{L^\infty_{t,r} } \to 0 \]
 and the $L^\infty$ profile selection. For the former, the Pythagorean expansion ensures that $\| \vec\gamma_{J,n}(0) \|_{H \times L^2}$ is a bounded sequence, hence this follows from the previous Lemma and
\[ \limsup_n \| \gamma_{J,n} \|_{S_\ell(\m R)} \to 0.  \]

For the latter, let $r_n$ such that
\[ |\psi_n(r_n) - \| \psi_n \|_{L^\infty} | \le \frac{1}{n}. \]
First assume $\liminf_n \| \psi_n \|_{L^\infty} >0$, and let $\e >0$ such that $ 2\e <\liminf_n \| \psi_n \|_{L^\infty}$. Choose $J$ so large that $\| \gamma_{J,n} \|_{L^\infty_{t,r}} \le \e$ for $n$ large enough. Then consider $j \le J$. If $j \notin \q J_0$, then $|t_{j,n}|/\lambda_{j,n} \to +\infty$ so that $\|  V_{j} \left( - \frac{t_{j,n}}{\lambda_{j,n}} \right) \|_{L^\infty} \to 0$. Hence
\[ |\psi_n(r_n)| \le \sum_{j \le J, j \in \q J_0} \left| V_j \left( 0,\frac{r_n}{\lambda_{j,n}} \right)  \right| + o_n(1) + \e. \]
In particular $\q J_0 \ne \emptyset$. As $\left| \ln \frac{\lambda_{j,n}}{\lambda_{k,n}} \right| \to +\infty$, we see that
\[ \lim_n \left\| \sum_{j \le J, j \in \q J_0} V_j \left( 0,\frac{r_n}{\lambda_{j,n}} \right)  \right\|_{L^\infty} = \max \{ \| V_j(0) \|_{L^\infty} \mid j \le J, j \in \q J_0\}. \]
This shows that
\[ \limsup_n   \| \psi_n \|_{L^\infty}  \le \sup \left\{ \| V_j(0) \|_{L^\infty} \mid j \in \q J_0 \right\}. \]

For the reverse inequality, first notice that as $j \to +\infty$, $\| V_j(0) \|_{H} \to 0$  (due to the Pythagorean expansion of the energy), and so $\| V_j(0) \|_{L^\infty} \to 0$. Hence there exists $j_0$ such that
\[ \sup \left\{ \| V_j(0) \|_{L^\infty} \mid j \le J, j \in \q J_0 \right\} = \| V_{j_0} \|_{L^\infty}. \]
Also, as $V_{j_0}(0)$ is continuous and tend to 0 at 0 and $+\infty$, there exists $r_0>0$ such that
\[ |V_{j_0}(0,r_0) |=  \sup \left\{ \| V_j(0) \|_{L^\infty} \mid j \in \q J_0 \right\}. \]
Consider the sequence $\psi_n(\lambda_{j_0,n} r_0)$. Then we have the expansion
\[ \psi_n(\lambda_{j_0,n} r_0) = V_{j_0}(0,r_0) + \sum_{j \le J, j \in \q J_0 \setminus \{ j_0 \} } V_j \left( 0,\frac{r_n}{\lambda_{j,n}} \right)  + \gamma_{J,n}(0, \lambda_{j_{0,n}} r_0) . \]
Again due to orthogonality of the profiles, we see that
\[ \liminf_n \| \psi_n \|_{L^\infty} \ge \liminf_n |\psi_n(\lambda_{j_0,n} r_0)| \to |V_{j_0}(0,r_0)|. \]
If $\liminf_n \| \psi_n \|_{L^\infty} =0$, then choosing a subsequence such that $\psi_{\sigma(n)} \to 0$ in $L^\infty$, and arguing as previously, we see that $\q J_0 = \varnothing$. Therefore $\| \psi_n - \gamma_{J,n} \|_{L^\infty} \to 0$ for all $J$, and so $\| \psi_n \|_{L^\infty} \to 0$. The desired equality also holds in this case.
\end{proof}

\begin{prop}[{Pythagorean expansion with cut-off, \cite[Corollary 8]{CKS13}}] \label{prop:pyth_cutoff}
We use the notation of the previous Proposition. Fix $J \ge 1$ and  Let $0 \le r_n \le s_n \le +\infty$ be two sequences. Then we have the expansion:
\begin{align*} 
\MoveEqLeft \|  \vec \psi_{n} \|_{H_\ell \times L^2(r_n \le r \le s_n)}^2 \\
& =  \sum_{j=1}^J \left\| \left( V_{j,\mr L} \left( - \frac{t_{j,n}}{\lambda_{j,n}}, \frac{r}{\lambda_{j,n}} \right), \frac{1}{\lambda_{j,n}} \partial_t V_{j,\mr L} \left( - \frac{t_{j,n}}{\lambda_{j,n}}, \frac{r}{\lambda_{j,n}} \right) \right) \right\|^2_{H_\ell \times L^2(r_n \le r \le s_n)}  \\
& \quad + \| \vec \gamma_{J,n}(0,r) \|^2_{H_\ell(r_n \le r \le s_n))} + o_{n}(1).
\end{align*}
\end{prop}

\begin{proof}
This is the content of \cite[Corollary 8]{CKS13}. The proof there is done for one sequence i.e. $r_n =0$. 
To derive the above expansion, it suffices to do the difference between the expansions with cut-off $r \le s_n$ and $r \le r_n$.
\end{proof}

We will use several times the following simple remark.

\begin{cor} \label{cor:psi_dispersion}
Let $\vec\psi_n$ be a bounded sequence in $H \times L^2$. Assume furthermore that $\| \psi_{0,n} \|_{L^\infty} \to 0$ and $\| \psi_{1,n} \|_{L^2} \to 0$. Denote $\vec \psi_{n,L}$ the linear evolution (to \eqref{LWl}) with data $\vec \psi_n$ at time $0$, then
\[ \| \psi_{n,L} \|_{S(\m R)} \to 0. \]
\end{cor}

\begin{proof}
It suffices to prove that any profile decomposition has no non-trivial profiles. Consider such a decomposition with profiles $(\vec V_{j,L})_j$ and parameters $(t_{j,n},\lambda_{j,n})$, and remainder $\vec \gamma_{J,n}$. Recall equipartition of the energy: if $|t_{j,n}|/\lambda_{j,n} \to +\infty$, then
\[ \left\| \frac{1}{\lambda_{j,n}} \partial_t V_{j} \left( -\frac{t_{j,n}}{\lambda_{j,n}} \right) \right \|_{L^2}^2 \to \frac{1}{2} \| \vec V_j \|_{H_\ell \times L^2}^2. \]
Hence, from Pythagorean expansion of the energy, it follows that all non-trivial profile $V_{j,L}$ must satisfy $t_{j,n} =0$ and $\partial_t V_{j,L}(0) =0$. But as $\| \psi_n \|_{L^\infty} \to 0$, equality \eqref{sup_psi_n} shows that $\| V_{j,L}(0) \|_{L^\infty} =0$, hence $V_{j,L} =0$ and $\vec V_{j,L} =0$. 

Therefore, all profiles $\vec V_{j,L}$ are trivial, and $\vec \psi_n$ is the remainder term $\vec \gamma_{J,n}$ (which does not depend on $J$), which satisfy the desired dispersion property.
\end{proof}

We recall the notion of nonlinear profile. If $\vec V$ is a solution of the linear equation \eqref{LWl} and $T \in \overline{\m R}$, there exists a unique wave map $U$ solution to \eqref{wm}, defined on a neighborhood of $T$ and satisfying $U(0)= U(\infty) =\ell$ and
\[ \| \vec U(t) - \vec V(t) - (\ell,0) \|_{H \times L^2} \to 0 \quad \text{as} \quad t \to T. \]
Notice that if $T=+\infty$, then $U$ is defined on some interval $[T_0,+\infty)$ and scatters at $+\infty$, i.e
\[ \| U  - \ell \|_{S([T_0, +\infty))} < +\infty. \]
This notion follows from local well posedness \cite[Theorem 2]{CKM08}; we refer to \cite{KM06,KM08} for further details.

\begin{prop}[Evolution of the decomposition] \label{prop:evol_decomp}
Let $\vec \psi_n$ be a sequence of wave maps such that 
\[ (\vec \psi_n(0) - (\ell,0))_n \text{ admits a profile decomposition  in the sense of  Theorem \ref{th:prof_decomp},} \]
from which we use the notations.

Denote $U_{j}$ nonlinear profiles associated to $\left( \vec V_{j}, - \lim_n \frac{t_{j,n}}{\lambda_{j,n}} \right)$. 
Let $t_n$ be such that
\[ \forall J \ge 1, \quad \frac{t_n - t_{j,n}}{\lambda_{j,n}} < T^+(U_j), \quad \text{and} \quad \sup_n \| U_j - \ell \|_{S([-\frac{t_{j,n}}{\lambda_{j,n}}, \frac{t_n - t_{j,n}}{\lambda_{j,n}}])} < + \infty. \]
Then for $n$ large enough, $T^+(\vec \psi_n) > t_n$, and for all $t \in [0,t_n]$,
\begin{align*}
\vec \psi_n(t) - (\ell,0) & =  \sum_{j=1}^J \left( \left( U_{j} \left( \frac{t - t_{j,n}}{\lambda_{j,n}}, \frac{r}{\lambda_{j,n}} \right), \frac{1}{\lambda_{j,n}} \partial_t U_{j} \left( - \frac{t_{j,n}}{\lambda_{j,n}}, \frac{r}{\lambda_{j,n}} \right) \right)  - (\ell,0) \right) \\\
& \quad + \vec \gamma_{J,n}(t,r) + \vec r_{n,J}(t,r),
\end{align*}
where $\limsup_{n \to +\infty} \| r_{J,n} \|_{S_\ell([0,t_n])} + \| \vec r_{J,n} \|_{L^\infty([0,t_n], H \times L^2)} \to 0$ as $J \to +\infty$.
\end{prop}

\begin{proof}
It is the translation of \cite[Proposition 2.8]{DKM11} via the transformation $\q T$. 
\end{proof}

Notice that $\vec \psi_n(0) - (\ell,0))_n$ is a bounded sequence in $H \times L^2$, in particular $\psi_n(0) = \psi_n(\infty) = \ell$. Let us emphasize that we can not evolve a decomposition where harmonic maps appear: this is a fundamental difference with the semi linear wave equation, where such stationary profile were allowed (they belong to $\dot H^1\times L^2$).

We will often use this result in the following particular case.

\begin{cor} \label{cor:psi_scat}
Let $\vec \psi_n$ be a sequence of wave maps such that $\vec \psi_n(\infty) = \vec \psi_n(0) =\ell$, and\begin{enumerate}
\item $\vec \psi_n(0) - (\ell,0)$ is a bounded sequence in  $H \times L^2$.
\item If $\vec \psi_{n,L} $ denotes the linear solution to \eqref{LWl} with initial data $\vec \psi_{n}(0) -(\ell,0)$, then $\| \psi_{n,L}  \|_{S(\m R)} \to 0$.
\end{enumerate}
Then for $n$ large enough $\vec \psi_n$ is defined globally on $\m R$ and
\[  \sup_{t \in \m R} \| \vec \psi_n(t) - (\ell,0) - \vec \psi_{n,L}(t) \|_{H \times L^2} \to 0 \quad \text{as} \quad n \to +\infty. \]
\end{cor}

\begin{nb}
Notice we can combine Corollary \ref{cor:psi_dispersion} with this last result, as the hypothesis of the latter are the conclusions given by the former.
\end{nb}

\begin{proof}
The second condition means that any profile decomposition is trivial (i.e) does not contain any non trivial profile. Hence with the notation of the previous Proposition, $\vec \psi_n(0) -(\ell,0) = \vec \gamma_{J,n}(0) =:  \vec \psi_{n,L}(0)$ (does not depend on $J$) and 
\[ \vec \psi_n(t) + (\ell,0) = \vec \psi_{n,L}(t) + r_n(t). \]
To conclude, we argue by contradiction: 
\begin{enumerate}
\item if $\vec \psi_n$ blow-up in finite time, consider $t_n = T^+(\vec \psi_n)+1$ (and similarly for $T^-(\vec \psi_n)$). 
\item if the convergence does not hold, there exists $\eta>0$ and $t_n$ be such that $\| r_n(t_n) \|_{H \times L^2} \ge \eta$.
\end{enumerate}
Each hypothesis contradicts Proposition \ref{prop:evol_decomp}.
\end{proof}

\section{Bubble decomposition for a sequence of wave maps}

Our goal here is to study a sequence of wave maps with vanishing (space-time) kinetic energy. First we prove that at any scale, there is local (strong) limit which is a harmonic map (possibly constant).

\begin{prop}[Profile at any scale] \label{prop:1prof}
We assume (A1)-(A2).

Let $A>0$ and $\vec \psi_n$ be wave maps be defined on the time interval $[-A,A]$, such that
\begin{enumerate}
\item $\vec \psi_n$ have uniformly bounded energy $E(\vec \psi_n) \le E$.
\item $\psi_n(0)$ is bounded.
\item For some sequence $r_n \to +\infty$, 
\[ \| \partial_t \psi_n \|_{L^2((-A,A), L^2(r \le r_n))} \to 0 \quad \text{as } n \to +\infty. \]
\end{enumerate}
Then $\vec \psi_{\sigma(n)}$, up to a subsequence $\sigma(n)$, converges to some harmonic map $(Q,0)$ locally strongly in the following sense: for any $R >0$,
\begin{equation} \label{eq:conv_1scale} 
\sup_{t \in [-A,A]} \| \vec \psi_{\sigma(n)}(t)  - (Q,0) \|_{H \times L^2([1/R,R])} \to 0 \quad \text{as } n \to +\infty.
\end{equation}
\end{prop}

\begin{nb}
Assume furthermore that for some sequences $t_n \subset [-A,A]$ and $r_n \to r_\infty \in (0,+\infty)$, 
\[ \psi_{\sigma(n)}(t_n,r_n)  \to l. \]
It follows from the convergence that $Q(r_\infty) = l$, hence from classification of harmonic maps, $Q$ is non constant if and only if $g(l) \ne 0$.
\end{nb}

\begin{nb}
A result of this type was already obtained in \cite{Str03}, but only for the first scale (and in the $L^2_{\mr{loc}}(H \times L^2([0,+\infty)_{\mr{loc}})$ topology), i.e. under the additional assumption that
\[  \sup_{t \in [-A,A]} E( \vec \psi_{n}; 0,1) \le \delta_0. \]
for some small $\delta_0 >0$.
\end{nb}

\begin{proof}
Up to rescaling, we can assume that $A=1$. 
As $E(\vec \psi_n)$ and $\psi_n(0)$ are bounded, $\vec \psi_n$ is bounded in $\q C([-1,1], (\dot H^1 \cap L^\infty) \times L^2)$. Therefore, up to a subsequence, there exists $\vec \psi_\infty$ such that
$\vec \psi_n \to \vec \psi_\infty$ a.e. $(-1,1) \times (0,+\infty)$, and $\psi_n \to \psi_\infty$ in $\q C_{\mr{loc}}([-1,1] \times [0,+\infty))$ and
\[ \vec \psi_n \stackrel{*}{\tendf} \vec \psi_\infty \quad L^\infty((-1,1),  \dot H^1 \times L^2) \text{ star-weakly}. \]
From Fatou Lemma, we deduce that
\[ \forall t \in [-1,1]-a.e.,  \quad \int \frac{g^2(\psi_\infty(t,r))}{r^2} rdr \le \liminf_n \int \frac{g^2(\psi_n(t,r))}{r^2} rdr. \]
It follows that $\vec \psi_\infty$ satisfies the wave map equation \eqref{wm} in a weak sense, and has finite energy $E(\vec \psi_\infty) \le \liminf_n E(\vec \psi_n)$. Now, as $\| \partial_t \psi_n \|_{L^2((-1,1), L^2(r \le r_n))} \to 0$, then for all $R \ge 0$, 
\[ \| \partial_t \psi_\infty \|_{L^2((-1,1), L^2(r \le R))} = 0, \]
and hence $\partial_t \psi_\infty =0$, that is $\vec \psi_\infty$  is a weak harmonic map: $r \partial_r (r \partial_r \psi) = f(\psi)$. It follows by elliptic regularity that $\vec \psi_\infty = (Q,0)$ a smooth harmonic map.

Let us prove as a first step that there holds strong local convergence in space-time $L^2_\loc((-1,1), H \times L^2)_\loc)$. Let $K$ be a compact of $(-1,1) \times (0, + \infty)$ (in space time). Then
\begin{equation}
\| \psi_n(t,r) - \psi_\infty(t,r) \|_{L^2(K,rdrdt)} + \| \partial_t \psi_n(t,r) \|_{L^2(K,rdrdt)} \to 0
\end{equation}
due the compact embedding $\dot H^1 \cap L^\infty (rdrdt) \to L^2(K, drdt/r)$ and to the vanishing of the kinetic energy. Fix now $\varphi \in \q D((-1,1) \times (0, + \infty))$.  We compute
\begin{align*}
\MoveEqLeft \int (\partial_r \psi_n - \partial_r \psi_\infty)^2(t,r) \varphi(t,r) rdrdt \\
& = - \int \Delta (\psi_n - \psi_\infty) (t,r) (\psi_n - \psi_\infty) (t,r)  \varphi(t,r) rdrdt \\
& \qquad  - \int \partial_r (\psi_n - \psi_\infty)(t,r)  (\psi_n - \psi_\infty)(t,r)   \partial_r \varphi(t,r)  rdrdt \\
& = - \int \partial_{tt} \psi_n(t,r) (\psi_n - \psi_\infty) (t,r)  \varphi(t,r) rdrdt \\
& \qquad - \int \frac{f(\psi_n(t,r)) - f(\psi_\infty(t,r))}{2r} (\psi_n - \psi_\infty) (t,r)  \varphi(t,r) drdt \\
& \qquad   - \int \partial_r (\psi_n - \psi_\infty)(t,r) (\psi_n - \psi_\infty) (t,r)  \partial_r \varphi(t,r) rdrdt \\
& = \int |\partial_t \psi_n(t,r)|^2 \varphi(t,r) rdrdt + \int \partial_t \psi_n (t,r) (\psi_n - \psi_\infty)(t,r) \partial_t \varphi(t,r) rdrdt  \\
& \qquad - \int \frac{f(\psi_n(t,r)) - f(\psi_\infty(t,r))}{2r} (\psi_n - \psi_\infty) (t,r)  \varphi(t,r) drdt \\
& \qquad - \int \partial_r (\psi_n - \psi_\infty)(t,r) (\psi_n - \psi_\infty)(t,r)   \partial_r \varphi(t,r) rdrdt
\end{align*}
Now, using again the compact embedding  $H_{\mr{loc}} (rdrdt) \to L^2_{\mr{loc}} (drdt/r)$, and the Cauchy-Schwarz inequality, we have
\begin{multline*}
\left| \int \partial_t \psi_n (\psi_n - \psi_\infty)(t,r) \partial_t \varphi (t,r) rdrdt \right| \\
\le \| \partial_t \psi_n \sqrt{|\partial_t \varphi|} \|_{L^2} \| (\psi_n - \psi_\infty) \sqrt{|\partial_t \varphi|} \|_{L^2} \to 0,
\end{multline*}
and similarly
\begin{align*}
\MoveEqLeft \left| \int \frac{f(\psi_n(t,r)) - f(\psi_\infty(t,r))}{2r} (\psi_n(t,r) - \psi_\infty(t,r))   \varphi(t,r) drdt \right| \\
& \le  \| (f(\psi_n) - f(\psi_\infty)) \varphi/r \|_{L^2} \| (\psi_n - \psi_\infty) \sqrt{\varphi/r} \|_{L^2} \\
& \le \omega_f(\sup_n \| \psi_n \|_{L^\infty})  \| (\psi_n - \psi_\infty) \sqrt{\varphi/r} \|_{L^2}^2 \to 0
\end{align*}
(We recall that $f$ is $\q C^1$, hence we can define $\omega_f(A) := \sup_{x,y \in [-A,A]} \frac{|f(x)-f(y)|}{|x-y|}$; we already noticed earlier that $\psi_n$ was a bounded sequence of $L^\infty_{t,x}$). Finally, as $\| \partial_r (\psi_n - \psi_\infty) \|_{L^2(rdrdt)}$ is bounded,
\begin{multline*}
\left| \int \partial_r (\psi_n - \psi_\infty)(t,r)  (\psi_n - \psi_\infty)(t,r)   \partial_r \varphi (t,r) rdrdt \right| \\
\le \| \partial_r (\psi_n - \psi_\infty)(t,r) \|_{L^2} \| (\psi_n - \psi_\infty)(t,r) |\partial_r \varphi(t,r)| \|_{L^2} \to 0.
\end{multline*}
This proves that
\[ \int \left( (\partial_r \psi_n - \partial_r \psi_\infty)^2(t,r) + \frac{(\psi_n -  \psi_\infty)^2(t,r)}{r^2} \right) \varphi(t,r) rdrdt \to 0. \]

Let us now prove \eqref{eq:conv_1scale}. Due to the previous convergence, the set
\[ \q P := \{ t \in [-1,1] \mid \forall R >0, \quad \| \vec \psi_n(t) - (Q,0) \|_{H \times L^2([1/(2R),R+1])} \to 0 \} \]
is dense in $[-1,1]$.

We now need the following version of uniform continuity of the flow around the harmonic map $(Q,0)$.

\begin{lem} \label{lem:Qflow}
Let $Q$ be a harmonic map, $T \ge 0$ and $\e >0$. There exist $\delta >0$ such that for all $0 \le r_1 \le 2r_1 < r_2 \le +\infty$ and wave map $\vec \psi$ such that
\[ \| \vec \psi(0) - (Q,0) \|_{H \times L^2([r_1,r_2])} \le \delta, \]
Then for all time $t \in (T^-(\vec \psi), T^+(\vec \psi))$ such that $|t| \le \min \{ T, (r_2-r_1)/2 \}$,
\[ \| \vec \psi(t) - (Q,0) \|_{H \times L^2([r_1+|t|,r_2-|t|])} \le \e. \]
\end{lem}

\begin{proof}
We postpone it to the Appendix.
\end{proof}

Fix $R > \sqrt 2$ (so that $2/R < R$) and let $\e>0$.

Define the integer $K$ to be the integer part of $2R+1$ so that $K \ge 2R$. Then for all $k \in \llbracket -K, K-1 \rrbracket$, there exist $t_k \in [k/K,(k+1)/K]$ such that $t_k \in \q P$. Let $\delta >0$ be provided by the previous Lemma \ref{lem:Qflow} with $T=1$ (and our previously fixed $\e>0$).

Define $N$ be such that 
\[ \forall k \in \llbracket -K, K-1 \rrbracket, \forall n \ge N, \quad  \| \vec \psi_n(t_k) - (Q,0) \|_{H \times L^2([1/(2R),R+1])} \le \delta. \]
Let $t \in [-1,1]$ and $n \ge N$. There exist $ k \in \llbracket -K, K-1 \rrbracket$ such that $|t-t_k| \le 1/K$. Hence by  \ref{lem:Qflow} (as $1/K \le 1$), 
\[ \| \vec \psi_n(t) - (Q,0) \|_{H \times L^2([1/(2R)+1/K,R+1-1/K])} \le \e. \]
Now $R+1-1/K \ge R$ and $1/(2R) + 1/K \le 1/R$; hence Lemma \ref{lem:Qflow} yields
\[ \forall t \in [-1,1], \quad  \| \vec \psi_n(t) - (Q,0) \|_{H \times L^2([1/R,R])} \le \e, \]
that is 
\[ \forall n \ge N, \quad  \sup_{t \in [-1,1]}  \| \vec \psi_n(t) - (Q,0) \|_{H \times L^2([1/R,R])} \le \e. \]
This is the desired convergence.
\end{proof}

We now prove a bubble decomposition result, which is the main result of this Section. It will be central in the soliton resolution in both the blow-up and global cases, and that we will also crucially use for the sharp scattering result (Theorem 
\ref{th:scat}).

\begin{thm}[Bubble decomposition] \label{th:profiles}
We assume (A1)-(A2).

Let $(\psi_n)_n \in \q C([-1,1], H \times L^2)$ be a sequence of wave maps. Assume that for some $R>0$,
\begin{enumerate}
\item The $\vec \psi_n$ have uniformly bounded energy $\sup_{n} E(\psi_n) = E < +\infty$.
\item For some $\ell \in \q V$, $\psi_n(0,R) \to \ell$ as $n \to +\infty$.
\item At $t=0$, the $\partial_t \psi_n$ have vanishing $L^2$-norm on $[0,R]$:
\[ \sup_{\lambda \le 1} \frac{1}{\lambda} \int_{-\lambda}^\lambda \int_0^{R} | \partial_t \psi_n(t,r)|^2 rdrdt \to 0 \quad \text{as} \quad n \to +\infty. \]
\item $\vec \psi_n(0)$ has vanishing energy on scale $1$: 
\[ \forall t \in [-1,1], \forall r >0,\quad  E(\vec \psi_n(t); r,R) \to 0 \quad \text{as} \quad n \to +\infty. \]
\end{enumerate}
Then there exists an integer $J \ge 0$, $J$ scales $(\lambda_{j,n})_n$ verifying
\[ 0 < \lambda_{J,n} \ll \cdots \ll \lambda _{2,n}  \ll \lambda_{1,n} \ll 1, \]
and $J$ harmonic maps $Q_j \in \q H$ such that, up to a subsequence $\vec \psi_{\sigma(n)}$,
\[ \vec \psi_{\sigma(n)}(t) - (\ell,0) = \sum_{j=1}^J \left(Q_{j} (\cdot/\lambda_{j,n}) - Q_j(\infty), 0 \right) + \vec b_n(t), \]
where  $\vec b_n \in \q C([-1,1], H\times L^2([0,R]))$ satisfies the following convergences. For all $A>0$,
\begin{enumerate}
\item (No energy at all scale) Let $\lambda_n$ be a sequence such that $0 \le \lambda_n \le R/A$. Then
\[ \sup_{t \in [-A \lambda_n , A \lambda_n]}  \| \vec b_n(t) \|_{H \times L^2( \lambda_{n}/A \le r \le A \lambda_{n})} \to 0. \]
\item (No energy up to the last scale) If $J \ge 1$, then 
\[ \sup_{t \in [-A \lambda_{J,n}, A \lambda_{J,n}]}  \| \vec b_n(t) \|_{H \times L^2(r \le A \lambda_{J,n})} \to 0 \quad \text{ as } n \to +\infty. \]
 If $J=0$, then $\sup_{t \in [-1/2, 1/2]}  \| \vec b_n(t) \|_{H \times L^2(r \le R)} \to 0$.
\item Let $\q C_A =\{ (t,r) \mid |t| \le \min(Ar, 1),\ r \le R \}$ be a truncated cone.  Then $\sup_{(t,r) \in  \q C_A} | b_n(t,r) | \to 0$. 
\end{enumerate}
Also
\begin{enumerate}
\item $\sum_{j=1}^J E(Q_{j},0) \le E$,
\item For all $1 \le j<J$, $Q_{j+1}(\infty) = Q_{j}(0)$, and $Q_1(\infty)=\ell$.
\end{enumerate}
\end{thm}

\begin{proof}
We first need to introduce some notation.

As $E(\psi_n)\le E$ and $\psi_n(0,R) \to \ell$, hence is bounded, there hold an $L^\infty$ bound on $\psi_n$: for some $K >0$,
\[ \forall n, \forall t, \forall r, \quad | \psi_n(t,r) | \le K. \]
By (A1), $\q V \cap [-K,K]$ is finite. For $\ell \in \q V$, denote $\ell^-, \ell^+ \in \q V$ the preceding and following points in $\q V$ (respectively), that is $\ell^- < \ell <\ell^+$ $(\ell^-, \ell^+) \cap \q V = \{ \ell \}$. Then define 
\[ \eta_\ell = \min\{ \sup \{ |g(x)| \mid x \in (\ell,\ell^+) \}, \sup \{ |g(x)| \mid x \in (\ell^-,\ell) \} \} >0, \]
and
\[ \delta_0 = \frac{1}{2}  \inf \{\eta_\ell \mid \ell \in V \cap [-K,K] \} > 0. \]
Finally, for $Q \in \q H$ a non constant harmonic map such that $Q(\infty) = \ell$ and normalized in the sense that
\[ Q(1) = \frac{1}{2} (Q(0) + Q(\infty), \]
let $\e_\ell>0$ be such that
\[ \forall r \le \sqrt{\e_\ell}, \quad |g(Q(r)| \le \delta_0/2, \quad \text{and} \quad |g(Q(1/r)| \le \delta_0/2, \]
and
\[ \e_0 = \inf \{ \e_\ell \mid \ell \in V \cap [-K,K] \} > 0. \]
Notice that by definition of $\e_0$, if $Q$ is a normalized harmonic map such that $Q(\infty) \in [-K,K]$ and $r_0$ is such that $|g(Q(r_0))| \ge \delta_0/2$, then 
\[ \text{if } r \le \e_0 r_0 \text{ or } r \ge r_0/\e_0, \quad |g(Q(r))| \le \delta_0/2. \]

Also notice that due to monotonicity of the energy along light cones, hypothesis \emph{(2)} and \emph{(4)} are in fact uniform in $t \in [-1,1]$. More precisely, there holds
\begin{itemize}
\item[\emph{(2')}] for all $r >0$,  $\sup_{t \in [-1,1]} \| \psi_n(t) - \ell \|_{L^\infty([r,R])} \to 0$ as $n \to +\infty$.
\item[\emph{(4')}] for all $r >0$,  $\sup_{t \in [-1,1]}  E(\vec \psi_n(t);r,R) \to 0 $ as $n \to +\infty$
\end{itemize}
Let us prove \emph{(4')} first. Let $r>0$ and $\delta>0$. For $r' \le r/2$ be such that $1/r' \in \m N$, apply \emph{(4)} with $t=kr'$, $k \in \llbracket -1/r'; 1/r' \rrbracket$. This gives $N$ such that for all $n \ge N$ and all $k \in \llbracket -1/r'; 1/r' \rrbracket$,
\[ E(\vec \psi_n(kr');r',R) \le \delta. \]
By monotonicity of the energy, with $\tau \in [-r',r']$ we get
\[ E(\vec \psi_n(kr'+\tau);r'+|\tau|,R) \le  E(\vec \psi_n(kr');r',R) \le \delta. \]
The $kr'+\tau$ cover all $[1,1]$, hence \emph{(4')}. 

We now turn to \emph{(2')}. Notice that for $\phi \in H$, 
\[ \| g(\phi) \|_{H}^2 \le (1+ \| g'(\phi) \|_{L^\infty}^2 ) E(\phi,0), \]
so that for all $t$  and $r>0$,
\[ \| g(\psi_n(t) ) \|_{H([r,R])}^2 \le (1+ \| g' \|_{L^\infty([-K,K])}^2) E(\vec \psi_n(t);r,R) \to 0, \]
where the convergence is uniform in $t$ due to \emph{(4')}. Due to Lemma \ref{lem:HLinfty}, we deduce that (for all $0 < r \le R/2$)
\[ \sup_{t \in [-1,1]} \| g(\psi_n(t)) \|_{L^\infty([r,R])} \to 0. \]
As $\q V \cap [-K,K]$ is finite and $\psi_n(0,R) \to \ell \in V$, it follows from a continuity argument that
\[ \forall r>0, \quad \sup_{t \in [-1,1]} \| \psi_n(t) -\ell \|_{L^\infty([r,R])} \to 0 \quad \text{as } n \to +\infty, \]
as desired.

\bigskip

\emph{Step 1:} Extraction of the profiles, and definition of $\lambda_{j,n}$ and $\vec b_n$.

We recall that an extraction is a function $\phi: \m N \to \m N$ which is (strictly) increasing.

We now define  the set of scales $S$ made of couples $\lambda := ((r_n)_n, \phi)$ where 
\[ (r_n)_n \in [0,R]^{\m N},  \quad \text{and } \phi: \m N \to \m N \text{ is an extraction}, \]
such that
\begin{equation} \label{ell_lambda}
\psi_{\phi(n)}(0,r_n) \text{  has a limit } l_\lambda \text{ such that } |g(l_\lambda)| = \delta_0.
\end{equation}
We denote by $S_0$ the subset of $S$ made of scales $((r_n)_n, \phi)$ such that for some non constant harmonic map $Q$, it satisfies furthermore
\begin{equation} \label{prof_lambda}
\forall A >0, \quad \| \vec \psi_{\phi (n)}(0) - (Q( r_{n}  \cdot),0) \|_{H \times L^2([r_{n}/A, A r_{n}])} \to 0.
\end{equation}

Notice that for all $r>0$, $\|\vec \psi_n(0) - (\ell,0) \|_{H \times L^2([r,R])} \to 0$ due to our hypothesis \emph{(4)}. Hence the pointwise bound gives
\[ \forall r >0, \quad  \| g(\psi_n(0)) \|_{L^\infty([r,R])} \to 0 . \]
It transpires that for any $\lambda = ((r_n)_n, \sigma) \in S$, then $r_n \to 0$. Also, in this case, if we define
\[ \phi_n(t,r) = \psi_{\sigma(n)}(r_{n} t, r_{n} r), \]
then $\phi_n$ is a wave map defined for $t \in [-1/r_n, 1/r_n]$ and for all $A >0$,
\[ \int_{-A}^A \int_0^{R/r_{n}} | \partial_t \phi_{n}(t,r)|^2 rdrdt = \frac{1}{r_{n}} \int_{-Ar_{n}}^{Ar_{n}} \int_0^{R} | \partial_t \phi_{\sigma(n)}(r_{n} t, r_{n} r)|^2 rdrdt \to 0, \]
and
\[ g(\phi_{n}(1)) = \lim_n g(\psi_{\sigma(n)}(r_{n})) = g(l_\lambda) \ne 0. \]
It follows from Proposition \ref{prop:1prof} (by rescaling again by a fixed factor $A$ and using a diagonal argument) that there exists an extraction $\pi$ and a non constant harmonic map $Q $, such that
\[ \forall A >0, \quad  \sup_{t \in [-A,A]}  \| \vec \phi_{\pi(n)}(t) - (Q, 0) \|_{H \times L^2([1/A, A])} \to 0. \]
In particular, due to continuity of the flow, we have that
\[ \forall t \in \m R, \quad \lim_{A \to +\infty} \lim_n E(\vec \phi_{\pi(n)}(t);1/A,A) = E(Q,0). \]
Unscaling, this can be rewritten as
\begin{equation} \label{eq:1prof}
\forall A >0, \quad  \sup_{t \in [-Ar_n,Ar_n]}  \| \vec \psi_{\phi \circ \pi(n)}(t) - (Q( r_{\pi(n)}  \cdot), 0) \|_{H \times L^2([r_{\pi(n)}/A, A r_{\pi(n)}])} \to 0.
\end{equation}
Notice that the scale $\tilde \lambda := ((r_{\pi(n)})_n, \phi \circ \pi) \in S_0$, we say it is adapted to $\lambda \in S$, and we call $Q$ the local limit at scale $\tilde \lambda$.

We now proceed with the extraction of the profiles. First assume that $S_0$ is empty. In this case, let us prove that for $n$ large enough, $\| g(\psi_n(0)) \|_\infty \le \delta_0$. Indeed, recall that $g(\psi_n(0,R)) \to 0$. If there exist an extraction $\phi$ such that for all $n$, $\| g(\psi_{\sigma(n)} (0)) \|_\infty > \delta_0$, then by continuity of $g(\psi_{\phi(n)})(0)$, for all $n$ there exist $r_n$ such that $|g(\psi_{\phi(n)})(0, r_n)| = \delta_0$. As $g^{-1}(\{ \pm \delta_0 \}) \cap [-K,K]$ is compact, up to extracting a subsequence, we can assume that $g(\psi_{\phi(n)})(0, r_n) \to \ell$ where $|g(\ell)| = \delta_0$. Hence $((r_n)_n, \sigma) \in S$ and we saw at the previous  paragraph how to construct an adapted scale to it: it follows that $S_0 \ne \varnothing$, a contradiction. 
Hence for $n \ge N$,
\[  \| g(\psi_n(0)) \|_\infty \le \delta_0. \]
Then we choose $J=0$ and $\vec b_n = \vec \psi_n$.

\medskip

If $S_0$ is not empty, we proceed by induction and construct a (finite) sequence of scales $\lambda_j = (r_{j,n})_n,\sigma_j) \in S_0$ for $j=1, \dots J$, such that
\begin{enumerate}
\item There exist an extraction $\pi_{j}$ such that $\sigma_{j+1} = \sigma_{j} \circ \pi_{j}$.
\item Orthogonality: $\ds \frac{r_{j+1,n}}{r_{j,\pi_{j+1}(n)}} \to 0$ as $n \to +\infty$.
\item For $j \in \llbracket 1, J-1 \rrbracket$ and all $n \in \m N$ and $r \in (r_{j+1,n}, \e_0 r_{j,\pi_j(n)} )$,  there holds $|g(\psi_{\sigma_{j+1}(n)}(0,r))| < \delta_0$.
\item For all $r \in [0,\e_0 r_{J,\pi_J(n)})$,  $|g(\psi_{\sigma_{J}(n)}(0,r))| < \delta_0$.
\end{enumerate}
The first two conditions give an order on $S_0$, the third one ensures that the scale $\lambda_{j}$ and $\lambda_{j+1}$ are ``consecutive'', and the fourth one is the stopping condition.

\medskip

Let $\phi_1$ be an extraction such that for all $n \in \m N$, $\| g(\psi_{\sigma_1(n)}(0)) \|_{L^\infty([0,R])} \ge \delta_0$. We can furthermore choose $\phi_1$ such that $\psi_{\sigma_1(n)}(r_{1,n})$ has a limit $l_1 \in g^{-1}(\{ \pm \delta_0 \})$, where $r_{1,n}$ is defined as follows: $r_{1,n}$ is such that $|g(\psi_n(r_{1,n})| = \delta_0$ and
\[  \forall r \in (r_{1,n},R], \quad |g(\psi_n(0,r))| < \delta_0. \]
$\phi_1$ and $r_{1,n}$ are well defined because $\psi_n$ is continuous, $g(\psi_n(R)) \to 0$ and $S_0$ is not empty.

Then $(r_{1,n},\sigma_1) \in S$, and let $\pi_1$ be adapted. This yields the first scale 
\[ \lambda_1:= ((r_{1,\pi_1(n)})_n,\sigma_1 \circ \pi_1) \in S_0, \]
and $Q_1$ is the local limit at scale $\lambda_1$.

Now assume that $(r_{j,n})_n,\sigma_{j})$ is constructed. For some non constant harmonic map $Q_j$ we have the convergence
\[ \forall A >0, \quad \| \vec \psi_{\sigma_j(n)}(0) - (Q_j (r_{j,n} \cdot),0)) \|_{H \times L^2(r_{j,n}/A, A r_{j,n})} \to 0 \quad \text{as } n \to +\infty. \]
For $A \ge 2$, the convergence also holds in $L^\infty(r_{j,n}/A, A r_{j,n})$, due to Lemma \ref{lem:HLinfty}. Now, as $|g(Q_j(\e_0))| \le \delta_0/2$, we see  (with $A = 1/\e_0$) that for large $n$, 
\[ |g(\psi_{\sigma_j(n)}(0, \e_0 r_{j,n}))| \le 2 \delta/3. \]

If for $n$ large enough, $\sup \{ |g(\psi_{\sigma_j(n)}(0,r)| \mid r \in [0,\e_0 r_{j,n}] \} \le \delta_0$, we stop here. Otherwise, we construct $(\tilde r_{j+1,n})_n,\sigma_{j+1})$ as follows. First choose an extraction $\rho_j$ such that denoting
\[ \phi_{j+1} = \sigma_j \circ \rho_j, \]
we have
\[ \forall n, \quad \sup_{ r \in [0,\e_0 r_{j,\rho_j(n)}]} |g(\psi_{\phi_{j+1}(n)}(0,r))| > \delta_0. \]
This allows to define $\tilde r_{j+1,n}$ such that $|g(\psi_{\phi_{j+1}(n)}(0, \tilde r_{j+1,n})| = \delta_0$ and 
\[ \forall r \in (\tilde r_{j+1,n},\e_0 r_{j,\rho_j(n)}), \quad |g(\psi_{\phi_{j+1}(n)}(r)| < \delta_0. \]
Extracting further, we can assume without loss of generality that $\psi_{\phi_{j+1}(n)}(0,\tilde r_{j+1,n})$ has a limit $l_{j+1}$, i.e. $(\tilde r_{j+1,n})_n, \phi_{j+1}) \in S$. Notice that due to the local convergence
\[ \vec \psi_{\sigma_j(\pi_j(n))}(0)) - (Q_j(r_{j,\rho_j(n)} \cdot),0) \]
on the scale $r_{j,n}$, it follows that
\[ \frac{\tilde r_{j+1,n}}{r_{j,\rho_j(n)}} \to 0 \quad \text{as } n \to +\infty. \] 

Let finally choose an extraction $\varpi_{j}$ adapted such that $(r_{j+1,n})_n, \sigma_{j+1}) \in S_0$, where $\sigma_{j+1} : = \phi_{j+1} \circ \varpi_{j}$ and $r_{j+1,n} := \tilde r_{j+1,\varpi_{j}(n)}$. Finally let $Q_{j+1}$ be the local limit at scale $(r_{j+1,n}, \sigma_{j+1})$.

Then $\pi_{j} := \rho_j \circ \varpi_{j}$ is an extraction such that
\[ \sigma_{j+1} = \sigma_j \circ \pi_j. \]
Also, we see that
\[ \frac{r_{j+1,n}}{r_{j,\pi_j(n)}} = \frac{\tilde r_{j+1,\varpi_j(n)}}{r_{j,\rho_j (\varpi_j(n))}} \to 0 \quad \text{as } n \to +\infty. \]
Finally, by the definition of $\tilde r_{j,+1,n}$, we have
\[ \forall r \in (r_{j+1,n},\e_0 r_{j,\pi_j(n)}), \quad |g(\psi_{\sigma_{j+1}(n)}(r)| < \delta_0. \]
Hence $(r_{j+1,n})_n, \sigma_{j+1})$ is as desired.

\bigskip

Notice that from our construction, we have $Q_{j+1}(\infty) = Q_j(0)$.

\bigskip

We know claim that this process has to stop after a finite number of steps. Indeed, fix $J \ge 1$ consider the sequence $(\vec \psi_{\sigma_J(n)})_n$. Then for $1 \le j \le J$, we have for all $A >0$,
\[ E(\vec \psi_{\sigma_j(n)}; 1/A r_{\pi_j \circ \cdots \circ \pi_j(n),j}, A r_{\pi_j \circ \cdots \circ \pi_J(n),j}) \to E(Q_j, 0; 1/A, A). \]
Now for $j < j'$, if we denote $\varpi = \pi_j \circ \cdots \circ \pi_{j'-1}$, and $\pi = \pi_{j} \circ \cdots \circ \pi_k$
\[ \frac{r_{\pi_{j'} \circ \cdots \circ \pi_k(n),j'}}{r_{\pi_j \circ \cdots \circ \pi_k(n),j}} = \frac{r_{\pi(n),j}}{r_{\varpi \circ \pi(n),j}} \to 0 \]
due to the fact that the scales $r_{n,j'}$ and $r_{\varpi(n),j}$ are orthogonal.

Summing this for $1 \le j \le J$, we thus get
\[ E \ge \liminf_n E(\vec \psi_{\sigma_k(n)}) \ge \sum_{j=1}^J E(Q_j, 0; 1/A, A).\]
Letting $A \to +\infty$, we get $E \ge  \sum_{j=1}^J E(Q_j, 0)$. As $Q_j(0), Q_j(\infty) \in [-K,K]$ for all $j$, we have that
\[ E(Q_j, 0) \ge E_K := \inf \{ G(k) - G(k') \mid k, k' \in V \cap [-K,K], k > k' \}. \]
$E_K >0$ due to assumption (A2). Hence $E \ge J E_K$ and this prove that the process has to stop after at most $E/E_K$ steps.

\bigskip

Thus we have constructed a sequence of $J$ scales $(r_{1,n}, \sigma_1), \dots, (r_{J,n},\sigma_J) \in S_0$ and of non constant harmonic maps $Q_j$ . We can now define for $j=1, \dots J$,
\[ \sigma = \sigma_J, \quad \lambda_{j,n} = r_{j, \sigma_{j+1} \circ \dots \circ \sigma_{J}(n)} \]
It will be convenient to write
\[ \lambda_{0,n} =R, \quad \lambda_{J+1,n} =0. \]
From our construction, we have that for all $j=1, \dots, J$, and $A >0$,
\[  \sup_{t \in [-A \lambda_{j,n}, A \lambda_{j,n}]} \| \vec \psi_{\sigma(n)}(t) - Q_j( \cdot / \lambda_{j,n}) \|_{H \times L^2(\lambda_{j,n}/A \le r \le A \lambda_{j,n})} \to 0 \quad \text{as } n \to +\infty, \]
and 
\[  \sup_{t \in [-1,1]}  \| \vec \psi_{\sigma(n)}(t) - \ell \|_{H \times L^2(\lambda_{0,n}/A \le r \le  \lambda_{0,n})} \to 0. \]
Also, for all $n \in \m N$ and $j=0, \dots, J$,
\begin{gather}
 \frac{\lambda_{j+1,n}}{\lambda_{j,n}} \to 0 \quad \text{as} \quad n \to +\infty \\
\| g(\psi_{\sigma(n)}(0)) \|_{L^\infty([\lambda_{j+1,n}/\e_0, \e_0 \lambda_{j,n}])} \le \delta_0, \label{eq:g(psi)}
\end{gather}

Define the error term: for $t \in [-1,1]$
\begin{equation} \label{def:bn}
\vec b_n (t)= \vec\psi_{\sigma(n)}(t) - (\ell,0) - \sum_{j=1}^J (Q_j( \cdot / \lambda_{j,n}) - Q_j(\infty)x, 0).
\end{equation}

\bigskip

\emph{Step 2 :} Convergence at all scales. Let $A >0$ and a sequence $\lambda_n \subset [0,R/A]$ be given. We now prove conclusion (1), that is
\begin{equation} \label{eq:bn0prof}
\sup_{t \in [-A\lambda_n, A\lambda_n]} \| \vec b_n(t) \|_{H \times L^2([\lambda_n/A,A\lambda_n])} \to 0 \quad \text{as } n \to +\infty.
\end{equation}
This means that $b_n$ has no profile in the cone $\q C_A$.

First let us prove
\begin{equation} \label{eq:bn(0)->0}
\| b_n(0) \|_{L^\infty([0,R])} \to 0 \quad \text{as } n \to +\infty
\end{equation}

From assumption  \emph{(4)}, for all $A >0$,
\[ \| \vec b_n(0) \|_{H \times L^2(R/A \le r \le R)} \to 0  \quad \text{as } n \to +\infty. \]
And it follows that from \emph{Step 1} and an easy induction that for all $j=1, \dots, J$, and $A >0$,
\[ \| \vec b_n(0) \|_{H \times L^2(\lambda_{j,n}/A \le r \le A \lambda_{j,n})} \to 0 \quad \text{as } n \to +\infty. \]
Hence due to Lemma \ref{lem:HLinfty},
\[  \| b_n(0) \|_{L^\infty(R/A \le r \le R)} +  \| b_n(0) \|_{L^\infty(r_{j,n}/A \le r \le A \lambda_{j,n})} \to 0 \quad \text{as } n \to +\infty. \]

We now argue by contradiction. Let $r_n \in [0,R]$ be such that $\limsup_n | b_n(0,r_n) | > 0$. 
From convergence on the scales $\lambda_{j,n}$ of $b_n(0)$, we have that
\[ \forall j=0, \dots J, \quad \frac{r_{\rho(n)}}{\lambda_{j,\rho(n)}} + \frac{\lambda_{j,\rho(n)}}{r_{\rho(n)}} \to +\infty. \]
We can assume that for some extraction $\rho$, and for some $j_0 \in \llbracket 0, J \rrbracket$,
\[ \lambda_{j_0+1,n} \ll  r_{\rho(n)} \ll \lambda_{j_0,n}. \]
For $j \le j_0$, $Q_{j}(r_{\rho(n)}/\lambda_{j,n}) \to Q_j(0) = Q_{j+1}(\infty)$ and for $j >j_0$, $Q_{j}(r_{\rho(n)}/\lambda_{j,n}) \to Q_j(\infty) = Q_{j-1}(0)$, hence
\[ \sum_{j=1}^J Q_j( r_{\rho(n)} / \lambda_{j,n}) - Q_j(\infty) \to \sum_{j=1}^{j_0} Q_{j+1}(\infty) - Q_j(\infty) = Q_{j_0}(0) - Q_1(\infty). \]
Now, as $Q_1(\infty) = \ell$, we deduce that
\[ b_n(0,r_{\rho(n)}) = \psi_{\sigma \circ \rho(n)}(r_n) - Q_{j_0}(0) + o_n(1). \]
 Up to extracting further we can assume, (recall $b_n(0)$ is continuous and $b_n(0,R) \to 0$), 
\[ b_{\rho(n)}(0,r_{\rho(n)}) = \e, \]
where $\e \ne 0$ is small so that we also have $g(\e+Q_{j_0}(0)) \ne 0$. It follows that
\[ \psi_{\sigma \circ \rho(n)}(0,r_n)) = b_{\rho(n)}(0,r_{\rho(n)})+ Q_{j_0}(0) + o_n(1) \to \e + Q_{j_0}(0). \]
Arguing as in \emph{Step 1}, (and relying on Proposition \ref{prop:1prof}), we deduce that there exist a harmonic map $Q$ such that $Q(1) = Q_{j_0}(0) + \e$ (in particular, $Q$ is not constant) and and extraction $\varpi$ such that
\[ \forall A >0, \quad \| \psi_{\sigma \circ \rho \circ \varpi (n)}(0) - Q \|_{H \times L^2([r_{\rho  \circ \varpi (n)}/A, A  r_{\rho \circ \varpi (n)}])} \to 0 \quad \text{as } n \to +\infty. \]
But then convergence also holds in $L^\infty$ and as $\| g(Q) \|_{L^\infty} \ge 2\delta_0$, and we deduce
\[ \liminf_{n \to +\infty} \| g(\psi_{\sigma(n)}(0)) \|_{L^\infty([\lambda_{j_0+1,n}, \lambda_{j_0,n}])} \ge 2 \delta_0, \] 
and we reached a contradiction with \eqref{eq:g(psi)}. This proves that for all sequences $(r_n)_n \subset [0,R]$, $\lim_n b_n(0,r_n) =0$, and hence
\[ \| b_n(0) \|_{L^\infty([0,R])} \to 0 \quad \text{as } n \to +\infty. \]
which is \eqref{eq:bn(0)->0}.

\bigskip

We now prove \eqref{eq:bn0prof} arguing by contradiction. Up to extracting a subsequence, we can assume without loss of generality the existence of $A>0$ and of a sequences $t_n$ such that $0 \le |t_n| \le A\lambda_n \le AR$, and for some $\rho \ge 0$ and  $\e > 0$ 
\[ \| b_n(t_n) \|_{H \times L^2( \lambda_n / A \le r \le A \lambda_n)} \ge \e, \quad \text{and} \quad \frac{t_n}{\lambda_n} \to \rho. \]
Also \emph{(2)} and \emph{(4)} show that 
\[ t_n, \lambda_n \to 0 \quad \text{as } n \to +\infty. \]
Similarly, due to \eqref{eq:1prof}, for all $j=0, \dots, J$,
\[ \frac{\lambda_n}{\lambda_{j,n}} + \frac{\lambda_{j,n}}{\lambda_n} \to \infty. \]
Up to extracting further we can assume that  there is $j_0 \in \llbracket 0, J \rrbracket$ such that for all $n$,
\[ \lambda_{j_0+1,n} \ll \lambda_n \ll \lambda_{j_0.n} ). \]
Now consider the wave map
\[ \phi_n(t,r) = \psi_{\sigma(n)}(\lambda_n t, \lambda_n r). \]
By inspection, $\phi_n$ is a finite energy wave map defined for times $t \in [-A,A]$,  and for $n$ large enough,
\begin{equation} \label{eq:Q_nonconstant}
\left\| \vec \phi_n \left( \frac{t_n}{\lambda_n} \right) - (Q_{j_0}(0),0) \right\|_{H \times L^2(1/A \le r \le A)} \ge \e/2,
\end{equation}
and
\[ \frac{1}{A} \int_{-A}^A \int_0^{R/\lambda_n} |\partial_t \phi_n (t,r)|^2 rdrdt = \frac{1}{A\lambda_n} \int_{-A \lambda_n}^{A \lambda_n} \int_0^{R} |\partial_t \psi_{\sigma(n)} (t,r)|^2 rdrdt \to 0. \]
Hence for some harmonic map $Q$, $\phi_n \to Q$ in the sense of Proposition \ref{prop:1prof}. $Q$ is not constant due to \eqref{eq:Q_nonconstant}. Unscaling, we have
\[ \forall R, \quad \sup_{t \in [-A\lambda_n,A\lambda_n]} \| \vec  b_n(t) - (Q(\cdot/\lambda_n),0) \|_{H \times L^2([\lambda_n/A,A \lambda_n])} \to 0 \]
But then, for any $t \in  [-A\lambda_n,A\lambda_n]$, $b_n(t,\lambda_n) \to Q(1) \ne 0$ as $n \to +\infty$: this contradicts \eqref{eq:bn(0)->0}. Hence \eqref{eq:bn0prof} is proved.

\medskip

Notice that it follows immediately from \eqref{eq:bn0prof} and Lemma \ref{lem:HLinfty} that
\begin{equation} \label{eq:step2}
\sup \{ |b_n(t,r)| \mid 0 \le r  \le R,  0 \le |t| \le \min\{ 1,A r \}  \} \to 0 \quad \text{as } n \to +\infty.
\end{equation}

\medskip

As $\| b_n(0) \|_{L^\infty([0,R])} \to 0$, and $\psi_{\sigma(n)} (0) \in \q V$, we see that for $n$ large enough,
\[ \psi_{\sigma(n)} (0) = Q_{J}(0)=:l. \]
Up to dropping the first terms of the sequence, we can assume that this holds for all $n$.

Then \eqref{def:bn} can be rewritten
\begin{equation} \label{def2:bn} 
\vec b_n(t) = \vec \psi_n(t) - (l,0) - \sum_{j=1}^J (Q_j(\cdot/\lambda_{j,n}) - Q_j(0), 0).
\end{equation}

\bigskip

\emph{Step 3:} (Convergence up to the last scale) We do the proof for $J \ge 1$, the proof in the case $J=0$ being completely similar. Fix $A >0$, we now prove that
\begin{equation} \label{eq:bn->0J}
\sup_{t \in [-A \lambda_{J,n}, A \lambda_{J,n}]} \| b_n(t) \|_{H \times L^2([0,A \lambda_{J,n}])} \to 0.
\end{equation}

From \emph{Step 2}, we know that there exists $\alpha_n \uparrow +\infty$ such that
\[ \sup_{t \in [-\alpha_n \lambda_{J,n}, \alpha_n \lambda_{J,n}]} \| b_n(t) \|_{H \times L^2([\lambda_{J,n}/\alpha_n,\alpha_n \lambda_{J,n}])} \to 0. \]
Let us first prove that
\begin{equation} \label{est:step3.1}
E(\vec \psi_{\sigma(n)}(0); 0 ; \alpha_n \lambda_{J,n}) \to 0.
\end{equation} 
We argue by contradiction. If the above convergence does not holds, there exists $\e>0$ and a subsequence that we still denote $\vec\psi_{\sigma(n)}$ such that for some $\mu_n \le \lambda_{J,n} / \alpha_n$, we have
\[ E(\vec \psi_{\sigma(n)}(0); 0 ; 2\mu_n) = \e >0. \]
By decreasing $\e>0$ if necessary, we can furthermore assume that 
\begin{equation} \label{def:e}
\forall x \in \m R, \quad |G(x) - G(l)| \le \e/2 \imp |f(x)- f'(l) (x-l)| \le \frac{f'(l)|x-l|}{2}.
\end{equation}
(we recall that $l := Q_J(0) \in \q V$, so that $f(l)=0$ and $f'(l) = g'(l)^2 >0$).
By monotonicity of the energy, we see that for $t \in [-\mu_n,\mu_n]$,
\[
E(\vec \psi_{\sigma(n)}(t); 0 ; \mu_n) \le \e \quad \text{and} \quad E(\vec \psi_{\sigma(n)}(t); 0 ; 3\mu_n) \ge \e.
\]
Let 
\[ \vec  u_n(t,r) = (\psi_{\sigma(n)}(\mu_n t, \mu_n r), \mu_n \partial_t \psi_{\sigma(n)}(\mu_n t, \mu_n r)). \]
Then $\vec u_n(t,r)$ is a wave map defined on the time interval $[-1,1]$, and for all $t \in [-1,1]$,
\begin{equation} \label{est:en_bn} 
E(\vec u_{n}(t); 0 ; 1) \le \e \quad \text{and} \quad E(\vec u_{n}(t); 0 ; 3) \ge \e.
\end{equation} 
The definition of $\e$ shows that for all $t \in [-1,1]$, $r \in [0,1]$
\begin{equation} \label{eq:f(un)}
|f(u_n(t,r))- f'(l) (u_n(t,r)-l)| \le \frac{f'(l)|u_n(t,r)-l|}{2}. 
\end{equation}
Also the condition \emph{(3)} yields
\begin{equation} \label{est:bnt}
\int_{-1}^{1} \int_0^{R/\mu_n} |\partial_t u_n(t,r)|^2 rdrdt \to 0 .
\end{equation}
This allows to apply Proposition \ref{prop:1prof}. As $\| u_n(0) - Q_J(0) \|_{L^\infty([0, \lambda_{J,n}/(\alpha_n \mu_n))} \to 0$, the local limit is constant, i.e. for all $B >0$,
\begin{equation} \label{est:bn1}
\sup_{t \in [-1,1]} \|  \vec u_n(t) - (l,0) \|_{H \times L^2([1/B,B])} \to 0.
\end{equation}
Let $\varphi : \m R \to \m R$ be an even cutoff function such that $\varphi(x) = 1$ for $|x| \le 1/2$ and $\varphi(x) =0$ for $|x| \ge 1$.
We can compute
\begin{align*}
\MoveEqLeft \iint |\partial_r u_n(t,r)|^2 \varphi(r) \varphi(t) rdrdt \\
& = - \iint (u_n(t,r)-l) \partial_{rr} u_n(t,r) \varphi(r) \varphi(t) rdrdt  \\
& \qquad - \iint (u_n(t,r)-l) \partial_{r} u_n(t,r) (r \varphi'(r) + \varphi(r)) \varphi(t) drdt \\
& = - \iint (u_n(t,r)-l) \partial_{tt} u_n(t,r) \varphi(r) \varphi(t) rdrdt \\
& \qquad - \iint \frac{(u_n(t,r)-l) f(u_n(t,r))}{r^2} \varphi(r) \varphi(t) rdrdt \\
& \qquad - \iint (u_n(t,r)-l) \partial_{r} u_n(t,r) \varphi'(r) \varphi(t) rdrdt \\
& = \int | \partial_t u_n(t,r)|^2 \varphi(r) \varphi(|t|) rdrdt - \iint \frac{(u_n(t,r)-l) f(u_n(t,r))}{r^2} \varphi(r) \varphi(t) rdrdt \\
& \qquad + \iint (u_n(t,r)-l) \partial_{t} u_n(t,r) \varphi(r) \varphi'(t) rdrdt \\
& \qquad - \iint (u_n(t,r)-l) \partial_{r} u_n(t,r) \varphi'(r) \varphi(t) rdrdt 
\end{align*}
Now from estimate \eqref{est:bnt},
\[ \int | \partial_t u_n(t,r)|^2 \varphi(r) \varphi(t) rdrdt = o(1). \]
Combining  \eqref{est:bnt} with the Cauchy-Schwarz inequality, we also have
\[ \iint (u_n(t,r)-l) \partial_{t} u_n(t,r) \varphi(r) \varphi'(t) rdrdt = o(1). \]
Also, as $\varphi'$ has support on $[1/2,1]$, estimate \eqref{est:bn1} and the Cauchy Schwarz inequality, we have
\[ \iint (u_n(t,r)-l) \partial_{r} u_n(t,r) \varphi'(r) \varphi(t) rdrdt  = o(1). \]
We now use that $f'(l) = g'(l)^2 > 0$. Then it follows from \eqref{eq:f(un)} that
\begin{multline*}
\iint \frac{(u_n(t,r)-l) f(u_n(t,r))}{r^2} \varphi(r) \varphi(t) rdrdt \\
\ge \frac{f'(l)}{2} \iint \frac{|u_n(t,r) -l|^2}{r^2} \varphi(r) \varphi(t) rdrdt.
\end{multline*}
Hence
\begin{align*}
 0 & \le \iint |\partial_r u_n(t,r)|^2 \varphi(r) \varphi(t) rdrdt  \\
 & \le - \frac{f'(l)}{2} \iint \frac{|u_n(t,r)-l |^2}{r^2}  \varphi(r) \varphi(t) rdrdt + o(1).
 \end{align*}
From this we deduce first that 
\[ \iint \frac{|u_n(t,r)-l |^2}{r^2}  \varphi(r) \varphi(t) rdrdt  \to 0, \]
then 
\[  \iint |\partial_r u_n(t,r)|^2 \varphi(r) \varphi(t) rdrdt \to 0. \]
Adding up the last 2 results along with \eqref{est:bnt}, we get
\[ \int_{-1/2}^{1/2} \| \vec u_n(t) - (l,0) \|_{H \times L^2(r \le 1/2)}^2 dt \to 0. \]
Now recalling \eqref{est:bn1}, we get
\[ \forall B >0, \quad \int_{-1/2}^{1/2} \| \vec u_n(t) - (l,0) \|_{H \times L^2(r \le B)}^2 dt \to 0. \]
This shows that
\[ \forall B>0, \quad \int_{-1/2}^{1/2} E(\vec u_n(t);0,B) dt \to 0. \]
However, this contradicts the second estimate in \eqref{est:en_bn}, and from there, estimate \eqref{est:step3.1} holds true.

\medskip

From \eqref{est:step3.1} it is now easy to prove \eqref{eq:bn->0J}. Let $A>0$. For $n$ large enough, $\alpha_n \ge 2A$. By finite speed of propagation, we deduce that
\[ \sup_{t \in [-A \lambda_{J,n}, A \lambda_{J,n}]} E(\vec \psi_{\sigma(n)}(t);0,A \lambda_{J,n}) \to 0. \]
By coercivity of the energy around $l = \psi_{\sigma(n)}(0)$, we deduce
\[ \sup_{t \in [-A \lambda_{J,n}, A \lambda_{J,n}]} \| \vec \psi_{\sigma(n)}(t) - (l,0) \|_{H \times L^2([0,A \lambda_{J,n}])} \to 0. \]
Notice that for all $1 \le j<J$, as $\lambda_{j,n} \ll \lambda_{j,n}$, we have
\[ \| (Q_j(\cdot/\lambda_{j,n}) - Q_j(0), 0) \|_{H \times L^2([0,A \lambda_{J,n}])} \to 0. \]
The last two statements and \eqref{def2:bn} yield \eqref{eq:bn->0J}.

In the case $J=0$, the same proof shows that
\[ \sup_{t \in [-1/2,1/2]} \| b_n(t) \|_{H \times L^2([0,1/2])} \to 0. \]
Now \emph{(4')} reads: for all $r \in (0,R)$,
\[ \sup_{t \in [-1,1]} \| b_n(t) \|_{H \times L^2([r,R])} \to 0. \]
We add up these last two statement to conclude the case $J=0$.
\end{proof}

\section{Scattering for wave maps below the \texorpdfstring{$L^\infty$}{Loo} threshold}

\begin{proof}[Proof of Theorem \ref{th:scat}]
Let  $\vec \psi$ be a finite energy wave map, with $\psi(\infty) =\ell$ and such that it satisfies \eqref{hyp:infty_bound}. Notice that $\psi(0) = \ell$ and there exists $c>0$ (depending only on $\delta_0$ and $\ell$) such that
\begin{equation} \label{gbound}
\forall t \in [0,T^+(\vec \psi)), \forall r \ge 0, \quad c |\psi(t,r) -\ell | \le  |g(\psi(t,r)| \le \frac{1}{c}  |\psi(t,r) -\ell |.
\end{equation}
From this point wise bound, we derive that 
\begin{equation} \label{Hbound}
\forall t \in [0,T^+(\vec \psi)), \quad \|  \vec \psi(t) - (\ell,0) \|_{H \times L^2} \le C E(\vec \psi).
\end{equation}
 
Also notice that $T^+(\vec \psi) = +\infty$. Assume indeed for the sake of contradiction that $T^+(\vec \psi) < \infty$. Due to \cite{Str03}, a bubble would form: hence for a sequence of time $t_n \uparrow T^+(\vec \psi)$, and of points $r_n$, $\psi(t_n,r_n) \to k$ where $k$ is such that $\ell, k$ are two consecutive elements of $V$. Thus $\liminf_n \| \psi(t_n) - \psi(\infty) \|_{L^\infty} \ge |k - \ell |  \ge d_\ell > c$, a contradiction.  

\medskip

We now do an induction on the energy in the spirit of the Kenig Merle concentration compactness argument  \cite{KM06,KM08}. Define $E_c$ to be the supremum of all $E \ge 0$ such that all wave maps $\vec \psi$ of energy $E(\vec \psi) \le E$, which satisfies $\psi(\infty)=\ell$ and \eqref{hyp:infty_bound}, are global and scatters.

Then \cite[Theorem 2]{CKM08} shows that $E_c >0$ (recall $|g'(\ell)| \in \{ 1, 2 \}$). We now argue by contradiction and assume that $E_c$ is finite. 

\bigskip

\emph{Step 1.} We first construct a critical element, that is a wave map $\vec V$ defined on $[0,+\infty)$, that satisfies \eqref{hyp:infty_bound}, but $\| V - \ell \|_{S_\ell([0,+\infty))} = +\infty$.

Let $\vec \psi_n$ be a minimizing sequence of wave maps, i.e. $\vec \psi_n$ satisfies \eqref{hyp:infty_bound} (hence $T^+(\vec\psi_n) = +\infty)$), and
\[ E(\vec \psi_n) \le E_c + \frac{1}{n}, \quad \text{and} \quad \| \psi_n - \ell \|_{S_\ell([0,+\infty))} = +\infty. \]
Up to rescaling, we can also assume that for all $n$
\[ E(\vec \psi_n(0);1,+\infty) = E_c/100. \]
Choose a sequence of time $t_n$ with vanishing $L^2$ norm of $\partial_t \psi_n(t_n)$. More precisely, we claim that for all $n$, there exist a sequence of times $t_{n,m}$ such that
\begin{enumerate}
\item $\ds \sup_{0 <  s \le t_{n,m}/2} \frac{1}{s} \int_{t_{n,m}-s}^{t_{n,m}+s} \int_0^{t_{n,m}/4} |\partial_t \psi_n(t,r) |^2 rdrdt \le 1/m$.
\item $\ds E(\vec \psi_n(t_{n,m}); t_{n,m}/m; t_{n,m}(1-1/m)) \le \frac{1}{m}$,
\item $\ds \| \psi_n(t_{n,m}) - \ell \|_{L^\infty(r \ge t_{n,m}/m)} \le \frac{1}{m}$,
\end{enumerate}
Indeed Corollary \ref{cor:psi_t->0} provides us with a sequence satisfying the first condition, and then, up to extracting, Proposition \ref{prop:sse} and Corollary \ref{cor:lc_infty} allows to satisfy the second and third condition (we emphasize that these last two results hold for any sequence).

Now choose $t_n =t_{n,n}$. Then Theorem \ref{th:profiles} applies to the sequence of wave maps $\vec\varphi_n(t,r) = \vec \psi_n(t_n+t_n t,t_n r)$, with $R=1/2$: indeed, we have by scaling
\[ \sup_{0 < \lambda \le 1} \frac{1}{\lambda} \int_{-\lambda}^\lambda \int_0^{1/2} |\partial_t \varphi_n(t,r)|^2 rdrdt \to 0, \]
and for $t \in [-1,1]$, and $r >0$, we have due to our second condition
\[ E(\vec \varphi_n(t);r,1/2) = E(\vec \psi_n(t_n); rt_n,t_n/2) \to 0. \]
Also notice that due to finite speed of propagation, we have
\begin{equation} \label{eq:psi_n_fsp}
\forall c>1, \quad \| \vec \psi_n(t_n) - (\ell,0) \|_{H \times L^2(r \ge ct_n)} \to 0.
\end{equation}
We can therefore apply Theorem \ref{th:profiles}. It yields a bubble decomposition which is trivial: it can not contain any harmonic map profile $Q_j$ due to \eqref{hyp:infty_bound}, and we see that for all $A>0$ and sequence $0 < \mu_n \ll t_n$,
\begin{equation} \label{eq:psi_n_no_scale}
\| \vec \psi_n(t_n) - (\ell,0) \|_{H \times L^2([\mu_n/A,A \mu_n])} \to 0, \quad \text{and} \quad
\| \psi_n(t_n) -\ell \|_{L^\infty(0,t_n/2)} \to 0.
\end{equation}
Condition (2) then translate into absence of profile on scale 1:
\begin{equation} \label{eq:psi_n_no_scale1}
\| \vec \psi_n(t_n) - (\ell,0) \|_{H \times L^2([t_n/n,(1-1/n)t_n])} \to 0. 
\end{equation}
Using the third condition, we also get
\begin{equation} \label{psi_n(t_n)->0}
\| \psi_n(t_n) - \ell \|_{L^\infty} \to 0.
\end{equation}
In view of \eqref{Hbound} $\vec \psi_n(t_n) -(\ell,0)$ is bounded in $H \times L^2$, hence admits (up to a subsequence) a profile decomposition in the sense of Theorem \ref{th:prof_decomp}. Denote $V_{j,L}$ the linear profiles, and $(t_{j,n}, \lambda_{j,n})$ the parameters.

We claim that there are no nontrivial profiles $V_j$ such that $t_{j,n}=0$. Consider indeed such a profile $V_j$, for the sake of contradiction. Up to extracting, and changing scale by a fixed factor, it suffices to rule out three cases:
\[ 1)\ \lambda_{j,n} \ll t_n, \quad 2)\ \lambda_{j,n} = t_n, \quad 3)\  \lambda_{j,n} \gg t_n. \]
In case $1)$, from \eqref{eq:psi_n_no_scale},  for any $A >1$, we have
\[ \| \vec \psi_n(t_n) - (\ell,0) \|_{H \times L^2(t\lambda_{j,n}/A \le r \le A \lambda_{j,n})} \to 0. \]
Recall the Pythagorean expansion with cut-off Proposition \ref{prop:pyth_cutoff}, it implies in particular that
\[  \left\| \left( V_{j,L} \left( 0, \frac{r}{\lambda_{j,n}} \right), \frac{1}{\lambda_{j,n} }\partial_t V_{j,L} \left( 0, \frac{r}{\lambda_{j,n}} \right) \right) \right\|_{H \times L^2 (\lambda_{j,n}/A \le r \le A \lambda_{j,n})} \to 0. \]
Unscaling, we get
\[ \| \vec V_{j,L}(0) \|_{H \times L^2([1/A,A])} \to 0. \]
As this is true for all $A >1$, we get $\vec V_{j,L}(0)=0$.

In case $3)$, from \eqref{eq:psi_n_fsp} and Proposition \ref{prop:pyth_cutoff}, we similarly get with $c=2$
\[ \left\| \left( V_{j,L} \left( 0, \frac{r}{\lambda_{j,n}} \right), \frac{1}{\lambda_{j,n} }\partial_t V_{j,L} \left( 0, \frac{r}{\lambda_{j,n}} \right) \right) \right\|_{H \times L^2 ( r \ge 2 t_n)} \to 0. \]
Unscaling, we deduce
\[ \| \vec V_{j,L}(0) \|_{H \times L^2(r \ge 2 t_n/\lambda_{j,n})} \to 0. \]
As $t_n/\lambda_{j,n} \to 0$, we get $\vec V_{j,L}(0)=0$.

In case $2)$, from \eqref{eq:psi_n_no_scale1} and Proposition \ref{prop:pyth_cutoff}, we get
\[ \left\| \left( V_{j,L} \left( 0, \frac{r}{t_n} \right), \frac{1}{\lambda_{j,n} }\partial_t V_{j,L} \left( 0, \frac{r}{t_n} \right) \right) \right\|_{H \times L^2 (t_n/n \le r \le t_n (1-1/n))} \to 0, \]
which after unsealing yields
\[ \| \vec V_{j,L}(0) \|_{H \times L^2([1/n,1-1/n])} \to 0.\]
Similarly, \eqref{eq:psi_n_fsp} and  Proposition \ref{prop:pyth_cutoff} give after unscaling
\[ \| \vec V_{j,L}(0) \|_{H \times L^2([c,+\infty))} \to 0. \]
Hence $\vec V_{j,L}(0) = 0$.

We ruled out all three cases, and this establish our claim that there is no profile with $t_{j,n}=0$.

Consider now the set $\q J$ of indices $j$ such that $\ds \frac{t_{j,n}}{\lambda_{j,n}} \to + \infty$ and the nonlinear profile $V_j$ does not scatter at $+\infty$ (or blows up in finite time). Any of these profile have energy greater or equal to $E_c >0$, hence (by Pythagorean expansion of the energy) there only is a finite number of them.

If  there is no such profiles then for all $j$, $\| V_{j} \|_{S(t \ge -t_{j,n}/\lambda_{j,n})}$ is bounded (it tends to 0 if $t_{j,n}/\lambda_{j,n} \to -\infty$). Then Proposition \ref{prop:evol_decomp} shows that there is a uniform bound $M$ such that, for all sequence $\tau_n \ge t_n$, 
\[ \| \psi_n - \ell \|_{S_\ell([t_n,\tau_n])} \le M. \]
This in turn implies that $\vec \psi_n$ scatters at $+\infty$: it is a contradiction. Hence $\q J \ne \varnothing$.

We can assume without loss of generality that $\q J$ is indexed by $1, \dots, J_0$.

Among such $j \le J_0$, choose $\lambda_{j,n}$ slowest, then among such $j$, we consider $j_0$ such that $t_{j_0,n}$ is lowest, i.e
\[ \forall j \le J_0, \quad \frac{\lambda_{j,n}}{\lambda_{j_0,n}} \to +\infty \quad \text{or} \quad \left( \lambda_{j,n}= \lambda_{j_0,n} \text{ and } \frac{t_{j_0,n} - t_{j,n}}{\lambda_{j_0,n}} \to -\infty \right). \]
$\vec V := \vec V_{j_0}$ will be our critical element. First let us show that $\vec V$ is global and satisfies \eqref{hyp:infty_bound}. Let $t \in [0, T^+(\vec V))$ and $r \ge 0$. Define $\tau_n = t_{j_0,n} + \lambda_{j_0,n} t$. 
Then
\[  \| V_{j_0,L} \|_{S_\ell([-\frac{t_{j,n}}{\lambda_{j,n}}, \frac{\tau_n - t_{j,n}}{\lambda_{j,n}}])} \to \| V_{j_0,L} \|_{S_\ell ((-\infty, t])} < +\infty. \]
By construction, for $j \in \q J \setminus \{ j_0 \} $, 
\[ \frac{\tau_n -t_{j,n}}{\lambda_{j,n}} \to - \infty \quad \text{hence} \quad \| V_{j,L} \|_{S_\ell([-\frac{t_{j,n}}{\lambda_{j,n}}, \frac{\tau_n - t_{j,n}}{\lambda_{j,n}}]} \to 0. \]
For $j \notin \q J$, notice that $\tau_n \ge 0$ for $n$ large enough ($t_n/\lambda_{j,n} \to +\infty$) and 
\begin{itemize}
\item if $\ds \frac{t_{j,n}}{\lambda_{j,n}} \to +\infty$, $\| V_{j,L} \|_{S_\ell(\m R)} < +\infty$, and 
\item if $\ds \frac{t_{j,n}}{\lambda_{j,n}} \to -\infty$, $\| V_{j,L} \|_{S_\ell([-\frac{t_{j,n}}{\lambda_{j,n}}, +\infty))} \to 0$.
\end{itemize}
Hence Proposition \ref{prop:evol_decomp} applies, we can evolve the profile decomposition up to $\tau_n$:
\begin{align*}
 \psi_n(\tau_n,\lambda_{j_0,n} r) -\ell & = V \left( t, r \right) - \ell + \sum_{j \ne j_0, j \le J} \left( V_{j} \left( \frac{\tau_n - t_{j,n}}{\lambda_{j,n}}, \frac{\lambda_{j_0,n}}{\lambda_{j,n}} r \right) - \ell \right) \\
 & \qquad + \gamma_{J,n}(\tau_n, \lambda_{j_0,n} r) + r_{J,n}(\tau_n, \lambda_{j_0,n} r).
 \end{align*}
By inspection (arguing as in the proof of equation \eqref{sup_psi_n}), we deduce that 
\[ |V(t,r)-\ell | \le c. \]
It follows that $\vec V_j$  satisfies \eqref{hyp:infty_bound}; as noticed above, we then have $T^+(\vec V) = +\infty$. Also, due to the Pythagorean expansion of the energy, we see that
\[ E(\vec V) \le E_c. \]
As we chose it so that it does not scatter at $+\infty$, we must have $E(\vec V) = E_c$ by definition of the critical energy: therefore $\vec V$ is a critical element.

\bigskip

\emph{Step 2.} We reach a contradiction.

For this, we can repeat the argument of \emph{Step 1} on $\vec V$:  there exist a sequence of times $t_n \uparrow +\infty$ such that
\begin{enumerate}
\item $\ds \sup_{0 <  s \le t_{n}/2} \frac{1}{s} \int_{t_{n}-s}^{t_{n}+s} \int_0^{t_{n}/4} |\partial_t V(t,r) |^2 rdrdt \to 0$ as $n \to +\infty$,
\item $\| V(t_{n}) -\ell \|_{L^\infty} \to 0$,
\item $\vec V(t_n) - (\ell,0)$ admits a profile decomposition, with profiles $\vec U_{j,L}$.
\end{enumerate}
Also, as the bubble decomposition yield no bubble, convergence up to the last scale shows that
\begin{equation} \label{eq:V_no_en_inside}
\| \vec V(t_n) - (\ell,0) \|_{H \times L^2 ([0,t_n/2])} \to 0.
\end{equation}
Arguing as in \emph{Step 1}, we see that one of the nonlinear profiles $\vec U_{j_0}$ is critical, in particular $E(\vec U_{j_0}) = E_c = E(\vec V)$. By Pythagorean expansion of the energy, it follows that there are no other nontrivial linear profiles in the profile decomposition of $\vec V(t_n)$, and that the dispersion term tends to 0 in $H\times L^2$. In short, there holds
\[ \vec V(t_n,r ) - (\ell,0) = \left( U_{j_0,L} \left( - \frac{t_{j_0,n}}{\lambda_{j_0,n}}, \frac{r}{\lambda_{j_0,n}} \right), \frac{1}{\lambda_{j_0,n}} \partial_t U_{j_0,L} \left( - \frac{t_{j_0,n}}{\lambda_{j_0,n}}, \frac{r}{\lambda_{j_0,n}} \right) \right) + o_n(1), \]
where the $o_n(1)$ is in $H \times L^2$.

\smallskip 

Assume $t_{j_0,n} =0$. Observe that $\vec V$ has energy on the light cone: from monotonicity of the energy along light cones, Proposition \ref{prop:sse} and \eqref{eq:V_no_en_inside}, we have
\[ \limsup_{n \to +\infty} \| \vec V(t_n) - (\ell,0) \|_{H \times L^2(|r-t_n| \ge A)} \to 0 \quad \text{as} \quad A \to +\infty. \]
In particular,
\[ \| \vec V(t_n) - (\ell,0) \|_{H \times L^2(|r-t_n| \ge t_n/2)} \to 0. \]
Now, as $\vec U_{j_0,L}(0) \ne 0$, let $\rho>0$ such that $\| \vec U_{j_0,L}(0) \|_{H \times L^2([\rho,2\rho])} =: \alpha_0 >0$. Then 
\[ \left\| \left( U_{j_0,L} \left(0, \frac{r}{\lambda_{j_0,n}} \right), \frac{1}{\lambda_{j_0,n}} \partial_t U_{j_0,L} \left( 0, \frac{r}{\lambda_{j_0,n}} \right) \right) \right\|_{H \times L^2([\rho \lambda_{j_0,n}, 2 \rho \lambda_{j_0,n}])} = \alpha_0 >0. \]
Comparing with $\vec V(t_n)$ we must have for $n$ large enough
\[ 2\rho \lambda_{j_0,n} \ge t_n/2 \quad \text{and} \quad \rho \lambda_{j_0,n} \le 3t_n/2. \]
Up to extracting, we can furthermore assume that $\lambda_{j_0,n}/t_n \to \lambda \in (0,+\infty)$.
But then, unscaling the concentration on $\vec V(t_n)$ we have
\[ \limsup_{n \to +\infty} \| \vec U_{j_0,L}(0) \|_{H \times L^2(|r-t_n|/\lambda_{j_0,n} \ge A/\lambda_{j_0,n})} \to 0 \quad \text{as} \quad A \to +\infty. \]
But this implies $\vec U_{j_0,L}(0) =0$, a contradiction.

\smallskip 

Assume now that $\ds \frac{t_{j,n}}{\lambda_{j,n}} \to +\infty$. Then $\| U_{j_0,L} \|_{S_\ell((-\infty;-t_{j,n}/\lambda_{j,n}])} \to 0$ so the same holds for the non linear profile: $\| U_{j_0} \|_{S_\ell((-\infty;-t_{j,n}/\lambda_{j,n}])} \to 0$. Then applying Proposition \ref{prop:evol_decomp} backward in time up to time $t =0 \le t_n$, we get
\[ \| V - \ell \|_{S_\ell(0,t_n)} \le  \| U_{j_0} - \ell \|_{S_\ell((-\infty;-t_{j,n}/\lambda_{j,n}])} + o(1) \to 0. \]
Hence by monotone convergence, we deduce $ \| U \|_{S_\ell([0,+\infty))} = 0$, and $V =0$, a contradiction.

\smallskip 

Assume finally that
$\ds \frac{t_{j,n}}{\lambda_{j,n}} \to -\infty$. Then $\| U_{j_0,L} \|_{S_\ell([-t_{j,n}/\lambda_{j,n},+\infty))} \to 0$ so the same holds for the non linear profile: $\| U_{j_0} - \ell \|_{S_\ell([-t_{j,n}/\lambda_{j,n},+\infty))} \to 0$. Then we can use Proposition \ref{prop:evol_decomp} to get that for any sequence $\tau_n \ge t_n$, 
\[ \| V  - \ell \|_{S_\ell([t_n,\tau_n))} \to 0. \]
This implies $\| V - \ell \|_{S_\ell([t_n,+\infty))} \to 0$: in particular $V$ scatters at $+\infty$, a contradiction.

\smallskip 

We reached a contradiction in all cases, hence $E_c = +\infty$.
\end{proof}

\section{Outside the light cone}

\begin{prop}[Scattering state] \label{prop:scat_state} 
We assume (A1)-(A2)-(A3').

Let $\vec \psi$ be a finite energy wave map such that $T^+(\vec \psi) =+\infty$. Denote $\ell = \psi(\infty)$.
There exist a map $\vec \phi_L$ solution to linear problem \eqref{LWl} and an increasing non-negative continuous function $\alpha(t)$ such that  $\alpha(t) = o(t)$ and
\[ \| \vec \psi(t) - (\ell,0) - \vec \phi_L(t)  \|_{H \times L^2(r \ge \alpha(t))} \to 0 \quad \text{as} \quad t \to +\infty. \]
\end{prop}

\begin{proof}
First we recall Proposition \ref{prop:sse}, hence it suffices to construct $\vec \phi_L$ such that for all $A \ge 0$,
\begin{equation} \label{eq:scat_state1}
\| \vec \psi(t)  - (\ell,0) - \vec \phi_L(t) \|_{H \times L^2(r \ge t-A)} \to 0 \quad \text{as} \quad t \to +\infty.
\end{equation}

The proof follows the scheme of \cite[Proposition 2.8]{CKLS13b}, except now it is more involved to obtain scattering for the approximations of $\psi$ around the light cone. We crucially rely on our new scattering result Theorem \ref{th:scat}.

Let $t_n \uparrow +\infty$ and define the sequence of wave maps $\vec \phi_n$ with data at time $t_n$ as a suitable extension in $\q H \times L^2$ of $\vec\psi(t_n)|_{r \ge t_n/2}$, as in Lemma \ref{lem:extH}. Specifically let $\vec \phi_n(t_n) = (\phi_{n,0}, \phi_{n,1})$ where
\begin{align*}
\phi_{n,0}(r) & = \begin{cases} \ds 
2 \frac{\psi(t_n,t_n/2) - \ell}{t_n} r + \ell & \text{if } r \le t_n /2, \\
\psi(t_n,r) & \text{if } r \ge t_n/2,
\end{cases} \\
\phi_{n,1}(r) & = \partial_{t} \psi(t_n,r).
\end{align*}
(Recall that $\ell = \psi(\infty)$). 
Then $\vec \phi_n (0) = \vec \phi_n(\infty) = \ell$ and from Lemma \ref{lem:extH},
\begin{equation} \label{phi_n_tn/2}
\| \vec \phi_n(t_n) -(\ell,0) \|_{H\times L^2(r \le t_n/2)} \to 0.
\end{equation}
By construction $\vec \psi$ and $\vec \phi_n$ coincide at time $t_n$ on $[t_n/2,+\infty)$, hence by finite speed of propagation, as long as they are defined at time $t$,
\begin{equation} \label{phi_n=psi}
\forall r \ge t_n/2+|t-t_n|, \quad  \vec \phi_n(t,r) = \vec \psi(t,r).
\end{equation}

\bigskip

\emph{Step 1:} Let us show that for $n$ large enough,
\begin{enumerate}
\item $\vec \phi_n$ is defined on $[A_0,+\infty)$ for some $A_0$ not depending on $n$.
\item $\vec \phi_n$ scatters at $+\infty$.
\end{enumerate}

Choose $\lambda=1/2$ in Proposition \ref{prop:sse}, and apply it to $\vec \psi$:
\[ \limsup_{t \to +\infty} E(\vec \psi(t); t/2,t-A)\to 0 \quad \text{as } A \to +\infty. \]
We then deduce that
\[ \limsup_n E(\vec \phi_n(t_n); 0,t_n -A) \to 0 \quad \text{as } A \to +\infty. \]
Let $E_\ell >0$ be the minimal energy of a non constant harmonic map $Q$ with $Q(\infty) = \ell$ (or if $\ell_\pm \in \q V \setminus \{\ell \}$ are the closest elements to $\ell$, with $\ell_- < \ell < \ell_+$, $E_\ell = 2 \min \{ G(\ell_+ ) - G(\ell), G(\ell) - G(\ell_-) \}$.
Choose $A_0$ large enough, so that for $n$ large enough, 
\[ E(\vec \psi(t_n); 0,t_n-A_0)  \le E_\ell/2, \]
By finite speed of propagation, we deduce that for all $\tau$ such that $\vec \phi_n(t+\tau)$ is defined
\[ E(\vec \phi_n(t_n+\tau); 0, t_n-A_0 - |\tau|) \le E(\vec \psi_n(t_n); 0,t_n-A_0) \le E_\ell/2. \]
We recall the blowup criterion derived in \cite{Str03}: blow up concentrates in the light cone an energy at least $E_\ell$. Hence for $|\tau| \le t_n - A_0$, this blow up criterion shows that $\vec \psi_n(t_n + \tau)$ is well defined: it then suffices to choose $\tau = A_0 - t_n$, and $T^-(\vec \phi_n) < A_0$. Of course we can drop the first terms, so that is holds for all $n$.

We now turn to scattering at $+\infty$. For this, we will prove that as $n \to +\infty$,
\begin{equation} \label{phi_n_infty}
\sup_{t \in [t_n, T^+(\phi_n))}  \| \phi_n(t) - \ell \|_{L^\infty} \to 0.
\end{equation}

By monotonicity of the energy outside cones, we see that for all $A \in \m R$, the limit
$\lim_{t \to +\infty} E(\vec \psi(t),t-A, +\infty)$ exists, let us denote it $\q E(A)$. As the energy density is non-negative, $\q E(A)$ is an increasing function of $A$, and is also bounded by $E(\vec \psi)$. Denote $\q E = \lim_{A \to +\infty}\q E(A) \le E(\vec \psi)$.
Let us show that
\begin{equation} \label{en_infty}
E(\vec \psi(t); t/2,+\infty) \to \q E.
\end{equation}
Indeed, let $\e >0$, and choose $A$ large so that 
\[ \q E -\e  \le \q E(A) \le \q E \quad \text{and} \quad \limsup_{t \to +\infty} E(\vec \psi(t); t/2, t-A) \le \e \]. 
There exist $T$ large such that
\[ \forall t \ge T, \quad \q E(A) \le E(\vec \psi(t); t-A,+\infty) \le \q E(A) +\e, \text{and} \quad E(\vec \psi(t); t/2,t-A) \le 2\e . \]
Then for all $t \ge T$, we have
\[ \q E - \e \le \q E(A) \le E(\vec \psi(t);t-A,+\infty) \le \q E(A) + \e \le \q E+\e. \]
Hence for $t \ge \max \{ T, 2A \}$,
\[ \q E - \e \le E(\vec \psi(t);t-A,+\infty) \le E(\vec \psi(t);t/2,+\infty) \le E(\vec \psi(t);t/2,t-A)  + \q E + \e \le \q E + 3 \e. \]
\eqref{en_infty} follows.

Now, from \eqref{en_infty} and \ref{prop:sse}, we have
\begin{equation} \label{psi_E}
\limsup_{t \to +\infty} | E(\vec \psi(t);t-A,+\infty) - \q E | \to 0 \quad \text{as } A \to +\infty.
\end{equation}
From \eqref{en_infty} and \eqref{phi_n_tn/2}, we deduce that
\begin{equation} \label{en_phi_n}
E(\vec \phi_n(t_n)) \to \q E \quad \text{as } n \to +\infty.
\end{equation}

We now prove \eqref{phi_n_infty}. Let $\e>0$.
From \eqref{psi_E}, there exist $A_1$ and $T_1$ such that
\[ \forall t \ge T_1, \quad | E(\vec \psi(t);t-A_1,+\infty) - \q E | \le \e. \]
We also use Corollary \ref{cor:lc_infty}. Let $T_2$ such that
\[\forall t \ge T_2, \quad \| \psi(t) - \ell \|_{L^\infty(r \ge t/2)} \le \e. \]
Define now $N$ such that $t_{N} \ge \max \{ T_1, 2A_1, T_2 \}$ and for $n \ge N$,
\[ |E(\vec \phi_n(t_n)) - \q E| \le \e. \]
Then for $n \ge N$, we have $t_n \ge t_{N}$ so that if $t \ge t_n$, then $t-A_1 \ge t-t_n/2 \ge t_n/2 + |t-t_n|$. By \eqref{phi_n=psi}, it transpires
\[ \forall t \in [t_n, T^+(\phi_n)), \quad | E(\vec \phi_n(t);t-A_1,+\infty) - \q E | \le \e. \]
By conservation of the energy,
\[  \forall t \in [t_n, T^+(\phi_n)), \quad E(\vec \phi_n(t);0,t-A_1) \le 2\e. \]
Due to the point wise bound, we get as $\phi_n(0) = \ell$ 
\[  \forall t \in [t_n, T^+(\phi_n)), \quad \| \phi_n(t) - \ell \|_{L^\infty(r \le t-A_1)} \le C \e. \]
Also if $t \ge t_n$, then $t \ge T_2$ and again by \eqref{phi_n=psi} and 
\[ \forall r \ge t-A_1, \quad | \phi_n(t,r) - \ell | = | \psi(t,r) - \ell | \le \e. \]
This proves that for $n \ge N$,
\[ \forall t \in [t_n, T^+(\vec \phi_n)), \quad \| \phi_n(t) -\ell \|_{L^\infty} \le 2\e, \]
which is exactly \eqref{phi_n_infty}.

Then Theorem \ref{th:scat} applies for all $n \ge N$, and shows that $T^+(\vec \phi_n) =+\infty$ and $\vec \phi_n$ scatters at $+\infty$. Also notice that \eqref{phi_n_infty} and \eqref{en_phi_n} show that for some $C>0$ and for $n$ large enough,
\begin{equation} \label{eq:bound_phi_n_H}
\sup_{t \ge t_n} \| \vec \phi_n(t) - (\ell,0)  \|_{H \times L^2} \le C\q E.
\end{equation}

\emph{Step 2:} Construction of $\vec \phi_L$ and end of proof.

Let $\vec \phi_{n,L}$ be the linear solution of \eqref{LWl} which is the scattering state of $\vec \phi_n$, that is, for all $n$,
\begin{equation} \label{eq:phin_scat}
\| \vec \phi_n(t) - (\ell,0) - \vec \phi_{n,L}(t) \|_{H \times L^2} \to 0 \quad \text{as } t \to +\infty.
\end{equation}
Recall that the flow of  \eqref{LWl} preserves the $H_\ell \times L^2$ norm. Along with the bound on $\| \vec \phi_n(t) \|_{H \times L^2}$, this shows that
\[ \| \vec \phi_{n,L}(0) \|_{H \times L^2} \le C \q E. \]
Up to extracting, we can assume that $\vec \phi_{n,L}(0)$ has a weak limit $\vec \phi_L(0)$ in $H \times L^2$. We define $\vec \phi_L(t)$ as the linear solution of \eqref{LWl} with initial data  $\vec \phi_L(0)$ at time 0.
 
Let $\tau_n$ be such that 
\[ \| \vec \phi_n(\tau_n) - (\ell,0) - \vec \phi_{n,L}(\tau_n) \|_{H \times L^2} \le \frac{1}{n}. \]
Up to extracting further, we can assume that the sequence $\vec \phi_{n}(\tau_n) - (\ell,0) -  \vec \phi_{L}(\tau_n)$, which is bounded in $H \times L^2$, admits a profile decomposition in the sense of Theorem \ref{th:prof_decomp}:
\begin{multline*} 
\vec \phi_{n}(\tau_n,r)  - (\ell,0) =  \vec \phi_L(\tau_n,r)  \\
 +  \sum_{j=2}^J 
 \left( \frac{1}{\lambda_{j,n}^{d-1}} V_{j,L} \left( - \frac{t_{j,n}}{\lambda_{j,n}}, \frac{r}{\lambda_{j,n}} \right), \frac{1}{\lambda_{j,n}^d} \partial_t V_{j,L} \left( - \frac{t_{j,n}}{\lambda_{j,n}}, \frac{r}{\lambda_{j,n}} \right) \right) + \gamma_{J,n}(0,r)
\end{multline*}
Notice that this appears as a profile decomposition for the sequence $(\vec \phi_{n,L}(\tau_n,r))$ with first profile $\vec \phi_{L}$ and parameter $t_{1,n} = \tau_n$, $\lambda_{1,n} =1$. Indeed, the  profile decomposition is constructed of via taking weak limits, that is:
\[ \vec V_{j,L}(0) \text{ is the weak limit of } \Lambda[\lambda_{j,n}] S_\ell(t_{j,n}) (\vec \phi_{n}(\tau_n) -(\ell,0)), \]
where $S(t)$ is the linear flow of \eqref{LWl} and $\Lambda$ is the scaling operator
\[ \Lambda[\mu] (\varphi_0, \varphi_1)(t,r) : = \left( \varphi_0 \left( \frac{t}{\mu}, \frac{r}{\mu} \right), \frac{1}{\mu} \varphi_1 \left( \frac{t}{\mu}, \frac{r}{\mu} \right) \right). \]
Also $\vec \phi_L(0)$ is the weak limit of  $S( -\tau_n) (\Lambda[1] \vec \phi_{n}(\tau_n) - (\ell,0)) = S( -\tau_n) (\vec \phi_{n}(\tau_n) - (\ell,0))$. Indeed endow $H \times L^2$ with the natural scalar product derived from the $H_\ell \times L^2$ norm:
\[ \left\langle \vec \rho, \vec \sigma \right\rangle = \int \left( \rho_1 \sigma_1 + \partial_r \rho_{0} \partial_r \sigma_0 +  g'(\ell)^2 \frac{\rho_0 \sigma_0}{r^2} \right) rdr \]
Then $S(t)$ is an isometry for $\langle, \rangle$ and if $\vec \varphi \in H \times L^2$,
\begin{align*}
\MoveEqLeft \left\langle S( - \tau_n) (\vec \phi_{n}(\tau_n) - (\ell,0)) , \vec \varphi \right\rangle  = \left\langle \vec \phi_{n}(\tau_n) - (\ell,0) , S( - \tau_n) \vec \varphi \right\rangle \\
& = \left\langle \vec \phi_{n,L}(\tau_n) , S( - \tau_n) \vec \varphi \right\rangle + O(1/n) =  \left\langle \vec \phi_{n,L}(0) ,  \vec \varphi \right\rangle + O(1/n) \to \left\langle \vec \phi_{L}(0) ,  \vec \varphi \right\rangle.
\end{align*}
 
We now proceed to prove \eqref{eq:scat_state1}; it suffices to show
\[ \forall A \ge 0, \quad \| \vec \psi(t) - (\ell,0) - \vec \phi_{L}(t) \|_{H \times L^2(r \ge t-A)} \to 0. \]
 Let $A \ge 0$ and $N$ be such that for $t_N \ge 2A$. Notice that in a similar fashion as previously, there hold the following profile decomposition for the sequence $(\vec \phi_{n}(\tau_n) - (\ell,0) - \vec \phi_{N,L}(\tau_n)$
 \begin{multline*} 
\vec \phi_{n}(\tau_n,r) - (\ell,0) - \vec \phi_{N,L}(\tau_n,r) =  \vec \phi_L(\tau_n,r) - \vec \phi_{N,L}(\tau_n,r) \\
 +  \sum_{j=2}^J 
 \left( \frac{1}{\lambda_{j,n}^{d-1}} V_{j,L} \left( - \frac{t_{j,n}}{\lambda_{j,n}}, \frac{r}{\lambda_{j,n}} \right), \frac{1}{\lambda_{j,n}^d} \partial_t V_{j,L} \left( - \frac{t_{j,n}}{\lambda_{j,n}}, \frac{r}{\lambda_{j,n}} \right) \right) + \gamma_{J,n}(0,r)
\end{multline*}
Now recall \eqref{eq:phin_scat} and \eqref{phi_n=psi}, so that
\begin{equation} \label{eq:psi_scat_loc}
\| \vec \psi(t) - (\ell,0) - \vec \phi_{N,L}(t) \|_{H \times L^2(r \ge t-t_N/2)} \to 0 \quad \text{as } t \to +\infty.
\end{equation}
As $\phi_n$ and $\psi$ coincide for $r \ge t-t_n/2$, hence for $r \ge t -t_N/2$, we get
\[ \| \vec \phi_n(t) - (\ell,0) - \vec \phi_{N,L}(t) \|_{H \times L^2(r \ge t-t_N/2)} \to 0 \quad \text{as } t \to +\infty. \]
Using the Pythagorean expansion with cut off (Proposition \ref{prop:pyth_cutoff}), we deduce that for all profiles in the above profile decomposition,  the $H \times L^2(r \ge t-t_n/2)$ (semi-) norm tends to $0$, and more specifically for the first profile, we get
\[  \| \vec \phi_L(\tau_n) - \vec \phi_{N,L}(\tau_n) \|_{H \times L^2(r \ge \tau_n - t_N/2)} \to 0 \quad \text{as } n \to +\infty. \]
As $\vec \phi_L(t) - \vec \phi_{N,L}(t)$ is solution to the linear solution \eqref{LWl}, using monotonicity of the $H_\ell \times L^2$ norm on outside cones, we deduce
\[ \| \vec \phi_L(t) - \vec \phi_{N,L}(t) \|_{H \times L^2(r \ge t - t_N/2)} \to 0 \quad \text{as } t \to +\infty. \]
Combined with \eqref{eq:psi_scat_loc}, we get
\[ \| \vec \psi(t) - (\ell,0) - \vec \phi_{L}(t) \|_{H \times L^2(r \ge t-t_N/2)} \to 0 \quad \text{as } t \to +\infty. \]
As $t_N/2 \ge A$, this proves \eqref{eq:scat_state1}, and the proof is complete.
\end{proof}

We now turn to the analogous result regarding blow up wave maps.

\begin{prop} \label{prop:reg_part}
We assume (A1)-(A2).

Let $\vec \psi$ be a finite energy wave map that blows up at time $T^+(\vec \psi)$. Then there exist $\ell : =  \lim_{t \uparrow T^+(\vec \psi)} \psi(t, T^+(\vec \psi) - t)\in \q V$ and a wave map $\vec \phi$ defined on a neighborhood of $T^+(\vec \psi)$ such that for $t < T^+ (\vec\psi)$ and $t \ge \max \{ 0, T^-(\vec \phi) \}$
\[ \forall r \ge T^+(\vec \psi) - t, \quad \vec \psi(t,r) = \vec \phi(t,r), \]
and
\[ \| \vec \phi (t) - (\ell,0) \|_{H \times L^2([0,T^+(\vec \psi) - t])} \to 0 \quad \text{as} \quad t \uparrow T^+(\vec \psi). \]
\end{prop}

\begin{proof}
\emph{Step 1.} Let us first construct $\vec \phi$. 

\begin{claim}
$\psi(t, T^+(\vec \psi) - t)$ has a limit $\ell \in \q V$ as $ t \uparrow T^+(\vec\psi)$.
\end{claim}

\begin{proof}
Indeed, we recall Proposition \ref{prop:sse} with $\lambda=1/2$:
\[ E(\vec \psi(t); (T^+(\vec \psi)-t)/2,T^+(\vec \psi)-t) \to 0 \quad \text{as} \quad t \uparrow T^+(\vec \psi). \]
In particular, we get
\[ \sup_{r \in [(T^+(\vec \psi)-t)/2,T^+(\vec \psi)-t]} | G(\psi(t,r)) - G(\psi(t,T^+(\vec \psi)-t)) | \to 0.\]
Also recall that for all $t$, $\| \psi(t) \|_{L^\infty} \le K$. As $G$ is a homeomorphism $\m R \to \m R$ from Assumption (A1)-(A2), we get that $G^{-1}$ is uniformly continuous on $[G(-K),G(K)]$ and from there,
\[ \sup_{r \in [(T^+(\vec \psi)-t)/2,T^+(\vec \psi)-t]} | \psi(t,r) - \psi(t,T^+(\vec \psi)-t) | \to 0. \]
Now we use again the vanishing of the energy in the self-similar region:
\[ \int_{(T^+(\vec \psi)-t)/2}^{T^+(\vec \psi)-t} |g(\psi(t,r)|^2 \frac{dr}{r} \to 0, \]
so that there exist $r(t) \in [(T^+(\vec \psi)-t)/2,T^+(\vec \psi)-t]$ such that $g(\psi(t,r(t)) \to 0$ as $t \uparrow T^+(\vec \psi)$. 

By the previous uniform convergence and continuity of $g$, we derive $g(\psi(t,T^+(\vec \psi)-t)) \to 0$.
Now $t \mapsto \psi(t,T^+(\vec \psi)-t)$ is continuous on $[0,T^+(\vec \psi))$. As $\q V = g^{-1}(\{0 \})$ is discrete, this implies that $\psi(t,T^+(\vec \psi)-t)$ has limit $\ell \in \q V$ as $t \to T^+(\vec \psi)$, as desired.
\end{proof}

Now consider any sequence of time $\tau_n \uparrow T^+(\vec \psi)$ and the wave maps $\vec \phi_n$ defined at time $\tau_n$ as follows 
\begin{align*}
\phi_n(\tau_n,r) & = \begin{cases}
\ds \ell + \frac{\psi(\tau_n,T^+(\vec \psi) - \tau_n) -\ell}{T^+(\vec \psi) - \tau_n} r & \text{if } 0 \le r \le T^+(\vec \psi) - \tau_n \\
\psi(\tau_n,r) & \text{if } r \ge T^+(\vec \psi) - \tau_n
\end{cases} \\
\partial_t \phi_n(\tau_n,r) & = \begin{cases}
0 & \text{if } 0 \le r \le T^+(\vec \psi) - \tau_n \\
\partial_t \psi(\tau_n,r) & \text{if } r \ge T^+(\vec \psi) - \tau_n. 
\end{cases}
\end{align*}
Let $m > n$. By definition and finite speed of propagation, we have
\[ \forall r \ge T^+(\vec \psi) - \tau_n - |\tau_n - \tau_m|, \quad \vec \phi_n(\tau_m,r) = \vec \phi_m(\tau_m,r), \]
hence, using again finite speed of propagation,
\[ \forall r \ge 2T^+(\vec \psi) - 2\tau_n, \quad \vec \phi_n(T^+(\vec \psi),r) = \vec \phi_m(T^+(\vec \psi),r), \]
It is then meaningful to define $\vec \phi$ to be the wave map with initial data at time $T^+(\vec \psi)$:
\begin{equation} \label{eq:def_phi}
 \vec \phi(T^+(\vec \psi) ,r) = \vec \phi_n(T^+(\vec \psi),r) \quad \text{for } r \ge  2T^+(\vec \psi) - 2\tau_n.
\end{equation}

\bigskip

\emph{Step 2.} Properties of $\vec \phi$.

Let $\q E(t) := E(\vec \psi(t); T^+(\vec \psi)-t,\infty)$. This is a non-decreasing function, so let us define here
\[ \q E := \lim_{t \uparrow  T^+(\vec \psi)} \q E(t). \]
Of course $\q E \le E(\vec \psi)$. Then by construction of $\phi_n$ and as $\phi_n(\tau_n,T^+(\vec \psi) - \tau_n) \to \ell$, there holds
\[ E(\vec \phi_n) \to \q E. \]
Hence by monotone convergence, we deduce that
\begin{multline*}
E(\vec \phi) = \lim_{n} E(\vec \phi(T^+(\vec \psi)); T^+(\vec \psi) - 2\tau_n, +\infty) \\
\le \lim_{n} E(\vec \phi_n(T^+(\vec \psi) - \tau_n); T^+(\vec \psi) - \tau_n, +\infty) = \q E,
\end{multline*}
that is $E(\vec \phi) \le \q E$. On the other side, 
\[ E(\vec \phi) =  E(\vec \phi(T^+(\vec \psi)); 0, +\infty) \ge \lim_{n} E(\vec \phi_n(T^+(\vec \psi) - \tau_n); T^+(\vec \psi) - \tau_n, +\infty) = \q E, \]
and finally we obtain that $\vec \phi$ has finite energy
\begin{equation} \label{eq:en_phi}
E(\vec \phi) = \q E.
\end{equation}

By definition of $\vec \phi$ \eqref{eq:def_phi} and finite speed of propagation, for all $t \in (T^-(\vec \phi), T^+(\vec\psi))$ we have:
\[ \forall r \ge 3T^+(\vec \psi) - t - 2\tau_n, \quad \vec \phi(t,r) = \vec \phi_n(t,r), \]
which yields for $t=\tau_n$
\[ \forall r \ge 3T^+(\vec \psi) -  3\tau_n, \quad \vec \phi(\tau_n,r) = \vec \phi_n(\tau_n,r) = \vec \psi(\tau_n,r), \]
Hence, again by finite speed of propagation, we conclude that for $t \in (T^-(\vec \phi), \tau_n]$
\[ \forall r \ge 3T^+(\vec \psi) - t - 2\tau_n, \quad \vec \phi(t,r) = \vec \psi(t,r), \]
Letting $n \to +\infty$, we finally obtain that for $t \in (T^-(\vec \phi), T^+(\vec\psi))$:
\begin{equation} \label{eq:phi=psi_olc}
\forall r > T^+(\vec \psi) - t, \quad \vec \phi(t,r) = \vec \psi(t,r).
\end{equation}

By continuity, for $t < T^+(\vec \psi)$, it also holds for $r = T^+(\vec \psi) - t$. In particular,
\[ \vec \phi(t,T^+(\vec \psi)-t) \to \ell \quad \text{as} \quad t \uparrow T^+(\vec \psi). \]

\bigskip

Now by definition of $\q E$, \eqref{eq:phi=psi_olc} implies that
\[ E( \vec \phi(t,r); T^+(\vec \psi) - t, +\infty) \to \q E = E(\vec \phi), \]
and by difference
\[ E( \vec \phi(t,r); 0, T^+(\vec \psi) - t) \to 0. \]
As $\vec \phi(t,T^+(\vec \psi)-t) \to \ell$, this implies 
\[ \sup_{r \in [0,T^+(\vec \psi)-t]} | \phi(t,r) - \ell | \to 0 \quad \text{as} \quad t \uparrow T^+(\vec \psi). \]
Therefore $\phi(0) =\ell$ and
\[ \| \vec \phi(t) - (\ell,0) \|_{H \times L^2([0,T^+(\vec \psi)-t])} \to 0. \qedhere \]
\end{proof}

\section{\texorpdfstring{$H \times L^2$}{HxL2} convergence for the bubble decomposition}

In this Section we consider again a sequence of wave maps, and improve the result of Theorem \ref{th:profiles} under the extra assumption (A3), that is 
\[ \forall \ell \in \q V, \quad g'(\ell) \in \{-1,1\}. \]
We show that that the error term in the bubble decomposition does in fact convergence to 0 in $H \times L^2$. This is the only step in this paper where we use (A3): it  guaranties that Proposition \eqref{prop:en_conc} holds for the linearized problem \eqref{LWl}. This Section is independent of Section 4 and 5.

\begin{prop} \label{prop:Hconv}
We assume (A1)-(A2)-(A3).

Let $\psi_n$ be a sequence of wave maps as in Theorem \ref{th:profiles}, and we use its notations.

We recall the existence of
harmonic maps $Q_1, \dots Q_J$ and scales $\lambda_{J,n} \ll \cdots \lambda_{1,n} \ll 1$ and denote
\[ \vec b_n(t,r) : = \vec \psi_n (t,r) - (\ell,0) - \sum_{j=1}^J (Q_j(r/\lambda_{j,n}),0). \]
Then
\[ \| \vec b_n(0) \|_{H \times L^2([0,R])} \to 0 \quad \text{as} n \to +\infty. \]
\end{prop}

\begin{proof}
By Theorem \ref{th:profiles}, we recall that $\vec b_n$ has vanishing energy at all scale, but we only use this property for the scales $\lambda_{0,n}, \dots \lambda_{J,n}$ here. More precisely, we will use the following convergence: there exists an increasing sequence $\alpha_n \to +\infty$ (and we also denote $\lambda_{0,n} =1$), such that 
\begin{enumerate}
\item $\| \vec b_n(t) \|_{L^\infty([-\alpha_n \lambda_{J,n}, \alpha_n \lambda_{J,n}], H \times L^2(r \le \alpha_n \lambda_{J,n}))} \to 0$, 
\item $\| \vec b_n(t) \|_{L^\infty([-\alpha_n\lambda_{j,n}, \alpha_n \lambda_{j,n}], H \times L^2(\lambda_{j,n}/\alpha_n \le r \le \alpha_n \lambda_{j,n}))} \to 0$ for all $j =1, \dots, J-1$,
\item $\| \vec b_n(t) \|_{L^\infty([-A,A], H \times L^2(r \ge \lambda_{0,n}/\alpha_n))} \to 0$,
\item $\| \partial_t b_n(0) \|_{L^2(r \le R)} \to 0$.
\item $\| b_n(0) \|_{L^\infty(r \le R)} \to 0$.
\item $\sup_{t \in [-A,A]} \| \vec b_n (t)\|_{H \times L^2(r \le R)}$ is bounded.
\end{enumerate}
In particular,
\begin{equation} \label{eq:bn_conv1}
\| b_n(0) \|_{L^\infty([0,R])} \to 0, \quad \text{and} \quad \| \partial_t b_n(0) \|_{L^2([0,R])} \to 0.
\end{equation}

We now argue by contradiction. Assume that $\| \vec b_n (0) \|_{H \times L^2([0,R])} \not\to 0$ as $n \to +\infty$. Up to extracting, we can assume that for some $\delta_0 > 0$,
\begin{equation} 
\forall n, \quad \| b_n(0) \|_{H([0,R])} \ge \delta_0.
\end{equation}
Due to convergences (1)-(2)-(3), and up to extracting further, there exists $j \in \llbracket 1, \dots, J \rrbracket$ and $\delta_1 >0$ such that
\begin{equation} \label{contradiction}
\| \vec b_n(0) \|_{H \times L^2([0, \alpha_n \lambda_{j,n}])} \to 0, \quad \text{and} \quad \forall n \in \m N, \quad \| b_n(0) \|_{H ([\alpha_n \lambda_{j,n}, \frac{\lambda_{j-1,n}}{\alpha_n}])} \ge \delta_1.
\end{equation}

Let $\beta_n \uparrow +\infty$ such that $\beta_n = o(\alpha_n)$, $\beta_n^2 = o(\lambda_{j-1,n}/\lambda_{j,n})$ and for all $k =0, \dots, J$,
\[ \beta_n = o ( \lambda_{k-1,n}/\lambda_{k,n}), \]
 as $n \to +\infty$.

We consider two times:
\[ \tau_{1,n}  := \beta_n \lambda_{j,n}, \quad \tau_{2,n} = \frac{\lambda_{j-1,n}}{\beta_n}. \]

We denote $m = Q_j(\infty)$: the linearized flow which will interest is now given by  (LW$_m$). Let 
$\vec b_{n,L}$ be the linear solution to (LW$_m$) with initial data at time 0:
\begin{align*}
b_{n,L}(0,r) & = \begin{cases}
b_n(0,r) & \text{if } r \le 4 \tau_{2,n} \\
\ds \frac{b_n(0,4 \tau_{2,n})}{\tau_{2,n}} (5\tau_{2,n} - r) & \text{if }  4 \tau_{2,n}  \le r \le  5 \tau_{2,n} \\
0 & \text{if } r \ge  5 \tau_{2,n},
\end{cases} \\
\partial_{t} b_{n,L}(0) & =0.
\end{align*}

\begin{claim}
We have
\begin{equation} \label{eq:psi_t1}
\| \vec \psi_n(\tau_{1,n} )  - (m,0) -  \vec b_{n,L}(\tau_{1,n} ) \|_{H \times L^2([\tau_{1,n}  , 3 \tau_{2,n}])} \to 0 \quad \text{as} \quad n \to +\infty.
\end{equation}
\end{claim}

\begin{proof}
By definition of $\vec b_n$ and (2), notice that
\[ \| \vec \psi_n(\tau_{1,n}  ) - (m,0) - \vec b_{n}(\tau_{1,n}  ) \|_{H \times L^2([\tau_{1,n}  , 3 \tau_{2,n}])} \to 0. \]
Hence we are led to compare $\vec b_{n}$ and $\vec b_{n,L}$.

First consider the interval $[\tau_{1,n}  , 2\tau_{1,n}  ]$. Due to finite speed of propagation, we have
\[ \| \vec b_{n,L}(\tau_{1,n}  ) \|_{H_\ell \times L^2( [\tau_{1,n}  , 2\tau_{1,n}  ])} \le \| \vec b_{n,L}(0) \|_{H_\ell \times L^2( [0, 3\tau_{1,n}  ])} \to 0 \]
by \eqref{contradiction}.
Using (2), we conclude
\begin{align*}
\MoveEqLeft \| \vec b_{n}(\tau_{1,n}  ) - (m,0) - \vec b_{n,L}(\tau_{1,n}  ) \|_{H \times L^2( [\tau_{1,n}  , 2\tau_{1,n}  ])} \\
& \le \| \vec b_{n}(\tau_{1,n} )  \|_{H_\ell \times L^2( [\tau_{1,n}  , 2\tau_{1,n}  ])} + \|  \vec b_{n,L}(\tau_{1,n}  ) \|_{H \times L^2( [\tau_{1,n}  , 2\tau_{1,n}  ])}  \to 0.
\end{align*}

We now work on $[2\tau_{1,n}, 3\tau_{2,n}]$.
Consider the wave maps $\vec \vartheta_n$ with initial data defined as follows
\begin{align*}
\vartheta_n(0,r) & = \begin{cases} 
\ds m + \frac{ \psi_n(0,\tau_{1,n}) - m}{ \tau_{1,n}  } r & \text{if } 0 \le r \le  \tau_{1,n}  \\
\ds \psi_n(0,r)  & \text{if } \ds \tau_{1,n}   \le r \le 4 \tau_{2,n} \\
\ds m + \frac{ \psi_n(0,4\tau_{2,n}) - m}{\tau_{2,n}} (5 \tau_{2,n} - r) & \text{if } \ds 4\tau_{2,n} \le r \le 5 \tau_{2,n} \\
\ds m& \text{if } \ds r \ge 5 \tau_{2,n}
\end{cases} \\
\partial_t \vartheta_n(0,r) & = \begin{cases} 
\partial_t \psi_n(0,r)  & \text{if } \ds \tau_{1,n}   \le r \le 4 \tau_{2,n} \\
0 & \text{otherwise.}
\end{cases}
\end{align*}
$\vec \vartheta_n(0)$ is an extension of $\vec \psi_n(0)$ with adequate affine reconnection. $\vec \vartheta_n$ coincide with $\vec \psi_n$ on $[\tau_{1,n}  , 4 \tau_{2,n}]$ at time 0, and hence at time $\tau_{1,n} \le \tau_{2,n}$ we have 
\begin{equation} \label{eq:psi=theta}
\forall r \in [2\tau_{1,n}  , 3 \tau_{2,n}], \quad
\vec \psi_n(\tau_{1,n}  ,r) = \vec \vartheta_n(\tau_{1,n}  ,r).
\end{equation}

Also $\vec \vartheta_n (0) = \vec \vartheta_n(\infty) = m$ and 
\[ \| \vartheta_n - m \|_{L^\infty} \to 0. \]
Because $ \psi_n(0,\tau_{1,n}  )$ and $ \psi_n(0,4 \tau_{2,n})$ tend to $m$, and with the equivalence of energy and $H \times L^2$ norm for $L^\infty$-small perturbation of $m \in \q V$ (Lemma \ref{lem:EH}), we infer that
\[ \vec \vartheta_n(0) \text{ is bounded in } H \times L^2. \]
And finally, we clearly have
\[ \| \partial_t \vartheta_n(0) \|_{L^2} \to 0. \]
Hence Corollary \ref{cor:psi_dispersion} combined with Corollary \ref{cor:psi_scat} allows to conclude that $\vec \vartheta_n$ is globally defined on $\m R$ and
\begin{equation} \label{eq:theta=lin}
\sup_{t \in \m R} \| \vec \vartheta_n(t) - (m,0) - \vec \vartheta_{n,L}(t) \|_{H \times L^2} \to 0,
\end{equation}
where $\vec \vartheta_{n,L}$ is the linear solution to (LW$_m$) with initial data $\vec \vartheta_{n,L}(0) = \vec \vartheta_{n}(0) - (m,0)$ at time $0$.

Now notice that $\vec \vartheta_{n,L} - \vec b_{n,L}$ is a solution to (LW$_m$) which also satisfies
\[ \| \vec \vartheta_{n,L}(0) - \vec b_{n,L}(0) \|_{H_m \times L^2} \to 0. \]
By conservation of the $H_m \times L^2$ norm, we deduce that
\begin{equation} \label{eq:theta=b_lin}
\| \vec \vartheta_{n,L}(\tau_{1,n}  ) - \vec b_{n,L}(\tau_{1,n}  ) \|_{H_m \times L^2} \to 0. \end{equation}
Combining \eqref{eq:psi=theta}, \eqref{eq:theta=lin} and \eqref{eq:theta=b_lin} yields
\[ \| \vec \psi_n(\tau_{1,n}  ) - (m,0) - \vec b_{n,L}(\tau_{1,n}  ) \|_{H \times L^2([\tau_{1,n}, 3 \tau_{2,n}])} \to 0. \]
This proves the claim.
\end{proof}

We now evolve up to time $\tau_{2,n}$.

\begin{claim} \label{claim:psi=bn_2}
We have
\[ \| \vec \psi_n(\tau_{2,n})  - (m,0) -  \vec b_{n,L}(\tau_{2,n}) \|_{H \times L^2([\tau_{2,n}, 2 \tau_{2,n}])} \to 0 \quad \text{as} \quad n \to +\infty. \]
\end{claim}

\begin{proof}
Let $\vec \varpi_{n}$ be the wave map with initial data at time $\tau_{1,n}$ as follows:
\begin{align*}
\varpi_{n}(\tau_{1,n} ,r) & = \begin{cases} 
\ds m + \frac{ \psi_n(\tau_{1,n} ,\tau_{1,n}  ) - m}{ \tau_{1,n} } r & \text{if } 0 \le r \le  \tau_{1,n}  \\
\ds \psi_n(\tau_{1,n} ,r)  & \text{if } \ds \tau_{1,n}   \le r \le 3 \tau_{2,n} \\
\ds m + \frac{ \psi_n(0,3\tau_{2,n}) - m}{\tau_{2,n}} (4 \tau_{2,n} - r) & \text{if } \ds  3 \tau_{2,n} \le r \le 4 \tau_{2,n} \\
\ds m & \text{if } \ds r \ge 4 \tau_{2,n}
\end{cases} \\
\partial_t \varpi_{n}(\tau_{1,n} ,r) & = \begin{cases} 
\partial_t \psi_n(\tau_{1,n} ,r)  & \text{if } \ds \tau_{1,n}   \le r \le 3 \tau_{2,n} \\
0 & \text{otherwise.}
\end{cases}
\end{align*}
Notice that $E(\vec \varpi_{n}(\tau_{1,n} );0, 3\tau_{1,n} ) \to 0$ (consider separately the intervals $[0,\tau_{1,n} ]$ and $[\tau_{1,n} ,3\tau_{1,n} ]$), hence $\vec b_{n,NL}$ is defined at least on the time interval $(-2\tau_{1,n} , 4\tau_{1,n} )$, and by monotonicity of the energy along cones
\begin{equation} \label{eq:bn2_conv1} 
E(\vec \varpi_{n}(0);0, 2\tau_{1,n} ) \to 0.
\end{equation}
Similarly, $E(\vec \varpi_{n}(\tau_{1,n} ); 3\tau_{2,n} - 2\tau_{1,n} , + \infty) \to 0$ because
$3 \tau_{2,n} - 2\tau_{1,n}  \ge 2 \tau_{2,n}$ (consider separately the intervals $[3 \tau_{2,n} - 2\tau_{1,n} , 3\tau_{2,n}]$, $[3 \tau_{2,n}, 4 \tau_{2,n}]$, and $[4 \tau_{2,n}, \infty)$). Hence, again by monotonicity of the energy along light cones,
\begin{equation} \label{eq:bn2_conv2} 
E \left(\vec \varpi_{n}(0); 3\tau_{2,n} - \tau_{1,n} , \infty \right) \to 0.
\end{equation}

Let us show that 
\begin{equation} \label{eq:bn2_conv3}
\| \varpi_{n}(0) - (m,0) \|_{L^\infty} + \| \partial_t \varpi_{n}(0) \|_{L^2} \to 0.
\end{equation}

Indeed, by finite speed of propagation,
\[ \forall r \in [2 \tau_{1,n} , 3 \tau_{2,n} - \tau_{1,n} ], \quad \vec \varpi_{n}(0,r) = \vec \psi_n(0,r). \]
Hence the convergence on this interval follows from \eqref{eq:bn_conv1} which implies 
\[ \| \vec \psi_n(0) -  (m,0) \|_{H \times L^2([2\tau_{1,n}, 3 \tau_{2,n}])} \to 0. \]
Therefore, \eqref{eq:bn2_conv3} follows  from this combined with \eqref{eq:bn2_conv1}, \eqref{eq:bn2_conv2} (recall the equivalence of the energy and the $H \times L^2$ norm under the assumption of small energy, Lemma \ref{lem:EH}).

Due to \eqref{eq:bn2_conv3}, we conclude from Corollary \ref{cor:psi_dispersion} combined with Corollary \ref{cor:psi_scat} that $\vec \varpi_{n}$ is global and if we denote $\vec \varpi_{n,L}$ the linear solution to (LW$_m$) with initial data $\vec \varpi_{n}(0) - (m,0)$ at time 0 :
\begin{equation} \label{eq:bn_disp}
\| \varpi_{n,L} \|_{S(\m R)} + \sup_{t \in \m R} \| \vec \varpi_{n}(t) - (m,0) - \vec \varpi_{n,L}(t) \|_{H \times L^2} \to 0 \quad \text{as} \quad n \to +\infty.
\end{equation}
Notice that at time $\tau_{1,n} $, we have
\begin{align*}
\MoveEqLeft \| \vec b_{n,L}(\tau_{1,n} ) - \vec \varpi_{n,L}(\tau_{1,n} ) \|_{H \times L^2(([\tau_{1,n}  , 3 \tau_{2,n}])} \\
& \le \| \vec b_{n,L}(\tau_{1,n} ) + (m,0) - \vec \psi_n(\tau_{1,n} ) \|_{H \times L^2([\tau_{1,n}  , 3 \tau_{2,n}])} \\
& \qquad + \| \vec \varpi_n(\tau_{1,n} ) - (m,0) - \vec \varpi_{n,L}(\tau_{1,n} ) \|_{H \times L^2(([\tau_{1,n}  , 3 \tau_{2,n}])} \\
& \to 0.
\end{align*}
Now $\vec b_{n,L} -  \vec \varpi_{n,L}$ is a solution to (LW$_m$): due to monotonicity of the $H_m \times L^2$ along cones, and as $[\tau_{2,n}, 2 \tau_{2,n}] \subset [\tau_{2,n}, 3 \tau_{2,n} -(\tau_{2,n} - \tau_{1,n})]$, we deduce
\[ \| \vec b_{n,L}(\tau_{2,n}) - \vec \varpi_{n,L}(\tau_{2,n}) \|_{H \times L^2([\tau_{2,n}, 2 \tau_{2,n}])} \to 0. \]
Hence, we get from \eqref{eq:bn_disp}
\[ \| \vec \varpi_{n}(\tau_{2,n}) - (m,0) - \vec b_{n,L}(\tau_{2,n}) \|_{H \times L^2([\tau_{2,n}, 2\tau_{2,n}])} \to 0. \]
To complete the proof, it suffices to notice that $\vec \varpi_n$ and $\vec \psi_n$ coincide at time $\tau_{1,n}$ on the interval $[\tau_{1,n}, 3\tau_{2,n}]$, so that at time $\tau_{2,n}$, they coincide on the interval
$[\tau_{2,n}, 2\tau_{2,n}]$.
\end{proof}

We can now easily reach a contradiction. Indeed from (2)-(3)-(4), equivalence of energy and $H \times L^2$  for $L^\infty$ perturbation (Lemma \ref{lem:EH}) and the definition of $\beta_n$, we have
\begin{align}
\MoveEqLeft\| \vec \psi(\tau_{2,n}) - (m,0) \|_{H \times L^2([\tau_{2,n}, 2 \tau_{2,n}])}^2 \nonumber \\
& \le C E( \vec \psi(\tau_{2,n}) - (m,0); \tau_{2,n}, 2 \tau_{2,n}) \nonumber \\
& \le C \sum_{k=1}^J E( (Q_k (\cdot/ \lambda_{k,n}),0) ; \tau_{2,n}, 2 \tau_{2,n}) + o(1) \nonumber \\
& \le C \sum_{k=1}^J E \left( (Q_k,0); \frac{\lambda_{j,n}}{\beta_n \lambda_{k,n}} , 2\frac{\lambda_{j,n}}{\beta_n \lambda_{k,n}} \right) + o(1) \nonumber \\
& \to 0. \label{psi->0_1}
\end{align}
However, due to Proposition \ref{prop:en_conc}, we have
\[ \| \vec b_{n,L} (\tau_{2,n}) \|_{H \times L^2([\tau_{2,n}, +\infty))} \ge \beta(1) \| b_{n,L} (0) \|_{H}  \ge \beta(1) \delta_1. \]
From the definition of $\vec b_{n,L}(0)$, we see that
\[ \| \vec b_{n,L}(0) \|_{H_\ell \times L^2([\tau_{2,n}, +\infty))} \to 0, \]
hence by monotonicity of the $H_m \times L^2$ norm along cones, we get
\[ \| \vec b_{n,L} (\tau_{2,n}) \|_{H \times L^2([2 \tau_{2,n}, +\infty))}  \to 0. \]
We can then conclude:
\begin{equation} \label{bn->1_1}
\| \vec b_{n,L} (\tau_{2,n}) \|_{H \times L^2([\tau_{2,n}, 2\tau_{2,n}])} \ge \beta(1) \delta_1 >0.
\end{equation}
Then \eqref{psi->0_1} and \eqref{bn->1_1} are in contradiction with Claim \ref{claim:psi=bn_2}. 

This shows that $\| b_n(0) \|_{H([0,R])} \to 0$, and so $\| \vec b_n(0) \|_{H \times L^2([0,R])} \to 0$, as desired.
\end{proof}

\section{Proof of the Main Theorem}

\subsection{Global case}

Let $\vec \psi$ be a finite energy wave map such that $T^+(\vec \psi) = +\infty$ and let $\ell: =\psi(\infty) \in \q V$. Then Proposition \ref{prop:scat_state} (and Proposition \ref{prop:sse}) provides us with a scattering state $\vec \phi_L$ and $\alpha(t) = o(t)$ such that:
\begin{equation} \label{eq:out_lc}
\| \vec \psi(t) - (\ell,0) - \vec \phi_L(t) \|_{H \times L^2(r \ge \alpha(t))} \to 0.
\end{equation}
Recall Proposition \ref{prop:light_cone_en_conc}, from which a weak version yields
\begin{equation} \label{eq:ss_lc}
\| \vec \phi_L(t)  \|_{H \times L^2([0,t/2])} \to 0.
\end{equation}
On the other side, Corollary \ref{cor:psi_t->0} yields a sequence $t_n$ such that
\[ \sup_{s, 0<s\le t_n/4} \frac{1}{s} \int_{t_n - s}^{t_n + s} \int_0^{t/2}  | \partial_t \psi(t,r)|^2 rdrdt \to 0 \quad \text{as} \quad n \to + \infty. \]
Consider now the sequence of wave maps
\[ \vec \psi_n(t,r) = ( \psi(t_n + t_n t, t_n r ), t_n \partial_t \psi(t_n + t_n t, t_n r), \]
Then for $R=1/2$ and $A=1/4$, the sequence $\vec \psi_n$ satisfies the conditions of Theorem \ref{th:profiles}, which we can combine with Proposition \ref{prop:Hconv}. Hence there exists harmonic maps $Q_1, \dots , Q_J$ and scales $\lambda_{J,n} \ll \cdots \ll \lambda_{1,n} \ll 1$ such that
\[ \left\| \vec \psi_n(0) - (\ell,0) - \sum_{j=1}^J (Q_j(\cdot/\lambda_{j,n}) - Q_j(\infty),0)  \right\|_{H \times L^2([0,1/2])} \to 0. \]
Unscaling, and denoting $\mu_{j,n} = \lambda_{j,n}/t_n$, we have
\begin{equation} \label{eq:in_lc}
\left\| \vec \psi(t_n) - (\ell,0) -  \sum_{j=1}^J (Q_j(\cdot/\mu_{j,n}) - Q_j(\infty),0) \right\|_{H \times L^2([0,t_n/2])} \to 0.
\end{equation}
Also, for all $j$, as $\mu_{j,n} \ll t_n$, $\| Q_j(\cdot/\mu_{j,n}) - Q_j(\infty) \|_{H([t_n/2,\infty])} \to 0$.
Combining this with \eqref{eq:out_lc}, \eqref{eq:ss_lc} and \eqref{eq:in_lc} yields
\[ \left\| \vec \psi(t_n) - (\ell,0) -  \sum_{j=1}^J (Q_j(\cdot/\mu_{j,n}) - Q_j(\infty),0)  - \vec \phi_L(t) \right\|_{H \times L^2} \to 0 \quad \text{as} \quad n \to +\infty. \]
This concludes the proof in the global case.

\subsection{Blow up case}

Let $\vec \psi$ be a wave map which blow-up in finite time $T^+(\vec \psi)$. Proposition \ref{prop:reg_part} show that there is a regular wave map $\vec \phi(t)$ defined on a neighborhood of $T^+(\vec \psi)$ such that for $t < T^+(\vec \psi)$ (and $t$ near enough $T^+(\vec \psi)$) there holds
\begin{equation} \label{eq:out_lc2}
\forall r \ge T^+(\vec \psi) -t, \quad \vec\psi(t,r) = \vec \phi(t,r).
\end{equation}
Also $\vec \psi(t,T^+(\vec \psi) -t) \to \ell \in \q V$ and
\begin{equation} \label{eq:phi_in_lc}
\| \vec \phi (t) -(\ell,0) \|_{H \times L^2([0,T^+(\vec \psi) -t])} \to 0 \quad \text{as} \quad t \uparrow T^+(\vec \psi).
\end{equation}
Now Corollary \ref{cor:psi_t->02} provides us with a sequence $t_n \uparrow T^+(\vec \psi)$ such that (we denote $\theta_n = T^+(\vec \psi) - t_n$)
\[ \sup_{s, 0<s\le \theta_n} \frac{1}{s} \int_{t_n - s}^{t_n + s} \int_0^{1-t}  | \partial_t \psi(t,r)|^2 rdrdt \to 0 \quad \text{as} \quad n \to + \infty. \]
Then sequence of wave maps
\[ \vec \psi_n(t,r) := ( \psi( t_n + \theta_n t, \theta_n r ), \theta_n \partial_t \psi(t_n + \theta_n t, \theta_n r), \]
 satisfies the conditions of Theorem \ref{th:profiles} with $R=1$ (and $A=1$) which we can combine with Proposition \ref{prop:Hconv}. Hence there exists harmonic maps $Q_1, \dots , Q_J$ and scales $\lambda_{J,n} \ll \cdots \ll \lambda_{1,n} \ll 1$ such that
\[ \left\| \vec \psi_n(0) - (\ell,0) - \sum_{j=1}^J (Q_j(\cdot/\lambda_{j,n}) - Q_j(\infty),0)  \right\|_{H \times L^2([0,1])} \to 0. \]
Let us unscale, and denote $\mu_{j,n} = \lambda_{j,n}/\theta_n$, and recall \eqref{eq:phi_in_lc} to deduce
\begin{equation} \label{eq:in_lc2}
\left\| \vec \psi(t_n) - \vec \phi(t_n) -  \sum_{j=1}^J (Q_j(\cdot/\mu_{j,n}) - Q_j(\infty),0) \right\|_{H \times L^2([0,\theta_n])} \to 0.
\end{equation}
Hence combining with \eqref{eq:out_lc2}, this yields
\[ \left\| \vec \psi(t_n) - \vec \phi(t_n) -  \sum_{j=1}^J (Q_j(\cdot/\mu_{j,n}) - Q_j(\infty),0) \right\|_{H \times L^2} \to 0 \quad \text{as} \quad n \to +\infty. \]
This settles the proof of Theorem \ref{th1} in the blow-up case.

\appendix

\section{Proof of  Lemma \ref{lem:Qflow}}

\begin{lem}[Energy and $H_\ell$ norm] \label{lem:EH}
Let $\ell \in \q V$, and $\psi$ be a function of finite energy on $[r_1,r_2]$, with $0 \le r_1 < r_2 \le \infty$.
\begin{enumerate}
\item $2 |G(\psi(r_2)) - G(\psi(r_1))| \le E((\psi,0);r_1,r_2)$. 
\item There exists $\delta_\ell>0$ and $C_\ell>0$ independent of $\psi$ such that if
\[ \| \psi - \ell \|_{L^\infty([r_1,r_2])} \le \delta_\ell, \]
then
\[  \frac{1}{C_\ell} \| \psi - \ell \|_{H([r_1,r_2])}^2 \le E((\psi,0);r_1,r_2) \le C_\ell \| \psi - \ell \|_{H([r_1,r_2])}^2. \]
\item If $\psi(r_1)=\ell$, and $E(\psi,0;r_1,r_2) \le \delta_\ell'$ then the hypothesis and the conclusion of (2) above hold.
\end{enumerate}
\end{lem}

\begin{proof}
These bounds are elementary, based on the Taylor expansion of $g$ and $G$ around $\ell$. We refer to \cite{C05} for the first result, and to \cite{CKLS13a} for the second and third ones.
\end{proof}

Our goal is now to prove Lemma \ref{lem:Qflow}. We start by a claim regarding the control of $L^\infty$ norm by $H$ even locally outside 0 or $\infty$.

\begin{lem} \label{lem:HLinfty}
There exists $c>0$ such that for any $0 \le r_1 < r_2 \le +\infty$ with $r_2 \ge 2r_1$, and $\phi \in H([r_1,r_2])$, then $\phi \in  \q C([r_1,r_2])$ and
\[ \| \phi \|_{L^\infty([r_1,r_2])} \le c \| \phi \|_{H([r_1,r_2])}. \]
\end{lem}

\begin{proof}
We focus on the case $0 < r_1$ and $r_2 <\infty$, as the other cases are well-known and simpler.
We recall that for $r \ge s$,
\[ | \phi(r) - \phi(s) | \le \int_s^r | \partial_r \phi(u) | du \le \sqrt{ \int_s^r | \partial_r \phi(u) |^2 u du } \sqrt{\int_s^r \frac{du}{u}} \le \| \phi \|_{H} \sqrt{\ln \frac{r}{s}}. \]
This proves continuity. 

Then let $r_0 \in [r_1,r_2]$ be such that $\| \phi \|_{L^\infty([r_1,r_2])} = | \phi(r_0)|$. If $\| \phi \|_{L^\infty([r_1,r_2])} \le \| \phi \|_{H([r_1,r_2])}$, there is nothing to prove. Assume the opposite, that is $\| \phi \|_{L^\infty([r_1,r_2])} \ge \| \phi \|_{H([r_1,r_2])}$. 

Then if $r \in [r_1,r_2]$ is such that 
\[ \frac{|r-r_0|}{r_0} \le \frac{1}{4} \le \frac{\| \phi \|_{L^\infty([r_1,r_2])}^2}{4\| \phi \|_{H([r_1,r_2])}^2}, \]
then $\ds \sqrt{\left| \ln \frac{r}{r_0} \right|} \le  \frac{\| \phi \|_{L^\infty([r_1,r_2])}}{2\| \phi \|_{H([r_1,r_2])}}$ and
\[ |\phi(r)| \ge |\phi(r_0)| -  | \phi(r) - \phi(r_0) | \ge | \phi(r_0))| - \| \phi \|_H  \frac{\| \phi \|_{L^\infty}}{2\| \phi \|_{H([r_1,r_2])}} \ge \frac{\| \phi \|_{L^\infty([r_1,r_2])}}{2}. \]
Now as $\ds \frac{r_2}{r_1} \ge 2$, then $\ds \frac{r_0}{r_1}$ or $\ds \frac{r_{2}}{r_0} \ge \sqrt 2$. Let us assume the latter (the former would be treated accordingly), that is
\[ \frac{r_2-r_0}{r_0} \ge \sqrt 2-1 > \frac{1}{4}. \]
Then 
\begin{align*}
 \| \phi \|_{H([r_1,r_2])}^2 & \ge \int_{r_0}^{r_2} \frac{\phi(r)^2}{r} dr \ge \int_{r_0}^{5r_0/4} \frac{\phi(r)^2}{r} dr \ge \frac{1}{4} \| \phi \|_{L^\infty([r_1,r_2])}^2 \int_{r_0}^{5r_0/4} \frac{dr}{r} \\
 & \ge \frac{\ln (5/4)}{4} \| \phi \|_{L^\infty([r_1,r_2])}^2. \qedhere 
 \end{align*}
\end{proof}

\begin{lem}[Extension in $H$] \label{lem:extH}
Let $c$ as in Lemma \ref{lem:HLinfty}, and $0 \le r_1 < r_2 \le +\infty$.

Let $\phi \in H ([r_1, r_2])$, then there exists $\psi \in H$ that extends $\phi$, that is, 
\[ \forall r \in [r_1,r_2], \quad \psi(r) = \phi(r), \]
and $\| \psi \|_H \le \| \phi \|_{H([r_1,r_2])} + 3 \| \phi \|_{L^\infty([r_1,r_2])}$.

If $r_2 \ge 2r_1$, the previous Lemma ensures $\| \psi \|_H \le (3c+1) \| \phi \|_{H([r_1,r_2])}$.
\end{lem}

\begin{proof}
Again, we only consider here the case $0 < r_1 < r_2 < +\infty$ and leave the other simpler cases to the reader.
We extend $\phi$ in the following way: let $r_1' = r_1/2$ and $r_2' = 2r_2$ and
\begin{align*}
\psi(r) & = \begin{cases}
0 & \text{if } r \le r_0' \\
\phi(r_1) \frac{r-r_1'}{r_1 -r_1'}  & \text{if } r_1' \le r \le r_1  \text{ (affine extension)}\\
\phi(r) & \text{if } r_1 \le r \le r_2 \\
\phi(r_2)  \frac{r_2' -r}{r_2' -r_2} & \text{if } r_2 \le r \le r_2'  \text{ (affine extension)}\\
0 & \text{if } r_2' \le r 
\end{cases}.
\end{align*}
Then $\psi$ is continuous, hence we deduce $\psi \in H$ and
\begin{align*}
\int_{r_1'}^{r_1} |\partial_r \psi(r)|^2 rdr & = \phi(r_1)^2 \frac{r_1^2 - {r_1'}^2}{2(r_1-r_1')^2}  = \frac{3}{2}  \phi(r_1)^2, \\
\int_{r_1'}^{r_1} \frac{\psi(r)^2}{r^2} rdr & \le \phi(r_1)^2 \int_{r_1'}^{r_1} \frac{dr}{r} \le  (\ln 2) \phi(r_1)^2, \\
\int_{r_2}^{r_2'} |\partial_r \psi(r)|^2 rdr & = \frac{3}{2} \phi(r_2)^2 , \\
\int_{r_2}^{r_2'} \frac{\psi(r)^2}{r^2} rdr & \le (\ln 2) \phi(r_2)^2.
\end{align*}
From this, we see that 
\[ \| \psi \|_{H([0,r_1])} + \| \psi \|_{H([r_2,+\infty))} \le 3 \| \phi \|_{L^\infty([r_1,r_2])}, \]
hence,  if we combine it with the previous Lemma when $r_2 \ge 2r_1$, we get
\[  \| \psi \|_{H} \le \| \phi \|_{H([r_1,r_2])} + 3 \| \phi \|_{L^\infty([r_1,r_2])} \le (3c+1) \| \phi \|_{H([r_1,r_2])}. \qedhere \]
\end{proof}

\begin{proof}[Proof of Lemma \ref{lem:Qflow}]
The proof relies on \cite[Corollary 2.4]{C05}, which we recall for the convenience of the reader:

\begin{lem}[{\cite[Corollary 2.4]{C05}}] \label{lem:Qflow1}
Let $T \ge 0$ and $\e>0$. There exist $\eta>0$ such that if the wave map $\vec \psi$ satisfies
\[ \| \vec \psi(0) - (Q,0) \|_{H \times L^2} \le \eta, \]
then $T^+(\vec \psi) > T$ and 
\[ \forall t \in [0,T], \quad \| \vec \psi(t) - (Q,0) \|_{H \times L^2} \le \e. \]
\end{lem}

Let $\delta = \eta/(3c+2)$, where $\eta$ is given by Lemma \ref{lem:Qflow1} and $c$ is as in Lemmas \ref{lem:HLinfty} and \ref{lem:extH}.
 
If we consider the restriction of $\psi(0) - Q$ to $[r_1,r_2]$, Lemma \ref{lem:extH} yields an initial data $\phi_0$ such that $\phi_0 - Q \in H$,  
\[  \forall r \in [r_1,r_2], \quad \phi_0(r) = \psi(0,r), \]
 and
\[ \| \phi_0(0) - Q \|_H \le C \| \psi_0(0) - Q \|_{H([r_1,r_2])} \le (3c+1) \delta. \]
Define
\[ \phi_1(0,r) = \begin{cases}
\partial_t \psi(0,r) & \text{if } r_1 \le r \le r_2, \\
0 & \text{otherwise,}
\end{cases}
\]
(observe that $\| \phi_1 \|_{L^2} \le \delta$). Let $\vec \phi$ be the wave map with initial data $(\phi_0, \phi_1)$: it has finite energy, coincide with $\vec \psi$ on $[r_1,r_2]$ at time $t=0$ and
\[ \| \vec \phi(0) - (Q,0) \|_{H \times L^2} \le (3c+2) \delta \le \eta. \]
Hence by Lemma \ref{lem:Qflow1}, $\vec \phi$ is defined at least up to time $T$ and if furthermore $|t| \le T$,
\[  \| \vec \phi(t) - (Q,0 ) \|_{H \times L^2([r_1+|t|,r_2-|t|])} \le \| \vec \phi(t) - (Q,0 ) \|_{H \times L^2} \le \e. \]
Also, by finite speed of propagation for any $t \in (T^-(\vec \psi), T^+(\vec \psi)$ such that $|t| \le (r_2-r_1)/2$,
\[ \forall r \in [r_1 + |t|, r_2 -|t|], \quad  \vec \phi(t,r) = \vec \psi(t,r). \]
Combining the last two properties proves Lemma \ref{lem:Qflow}.
\end{proof}

\small


\bigskip
\bigskip

\textsc{Raphaël Côte}\\
CNRS and École Polytechnique \\
Centre de Mathématiques Laurent Schwartz UMR 7640 \\
Route de Palaiseau, 91128 Palaiseau cedex, France \\
\email{cote@math.polytechnique.fr}

\end{document}